\documentclass{article}

\usepackage{arxiv}

\usepackage[normalem]{ulem}
\usepackage[utf8]{inputenc} %
\usepackage[T1]{fontenc}    %
\usepackage{hyperref}       %
\usepackage{url}            %
\usepackage{booktabs}       %
\usepackage{amsfonts}       %
\usepackage{nicefrac}       %
\usepackage{microtype}      %
\usepackage{lipsum}		%
\usepackage{epstopdf}
\usepackage{doi}
\usepackage{parskip}
\usepackage{color}
\usepackage{caption}
\usepackage{subcaption}
\usepackage{float}
\usepackage{mathematics}
\usepackage{graphics,graphicx}
\usepackage{wrapfig}
\usepackage{amsmath}
\usepackage{cancel}
\usepackage{enumitem}
\usepackage{todonotes}
\usepackage{soul}
\usepackage{dsfont}
\usepackage{verbatim}
\usepackage{algpseudocode}
\usepackage{algorithm}
\usepackage{mathtools}
\usepackage{tcolorbox}
\usepackage{scalerel}
\usepackage{amssymb}

\DeclareFontFamily{OT1}{pzc}{}
\DeclareFontShape{OT1}{pzc}{m}{it}{<-> s * [1.10] pzcmi7t}{}
\DeclareMathAlphabet{\mathpzc}{OT1}{pzc}{m}{it}

\algnewcommand\And{\textbf{and}}

\DeclarePairedDelimiter{\ceil}{\lceil}{\rceil}

\usepackage{amsthm}
\newtheorem{lemma}{Lemma}
\newtheorem{remark}{Remark}

\newtheorem{assumption}{Assumption}
\newtheorem{problem}{Problem}
\newtheorem{theorem}{Theorem}

\newcommand{\cov}{\mathcal{C}}
\newcommand{\MC}[1]{\mathcal{#1}}

\newcommand{\Vx}{\Vec{x}}
\newcommand{\Vy}{\Vec{y}}
\newcommand{\psih}{\widehat{\psi}}
\newcommand{\Psih}{\widehat{\Psi}}
\newcommand{\Thetah}{\widehat{\Theta}}
\newcommand{\WH}[1]{\widehat{#1}}
\newcommand{\KL}{Karhunen–Lo\`eve~}

\newcommand{\Lhood}{\mathcal{L}}
\newcommand{\yobs}{\Vec{y}_{obs}}
\newcommand{\eff}{\mathrm{eff}}

\newcommand{\etal}{\textit{et al.}}

\title{Hierarchical Gaussian Random Fields for Multilevel Markov Chain Monte Carlo: Coupling Stochastic Partial Differential Equation and The Karhunen-Loève Decomposition}

\author{ 
\href{https://orcid.org/0000-0001-6882-9737}{\includegraphics[scale=0.06]{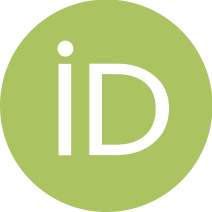}\hspace{1mm}
    Sohail Reddy \thanks{Corresponding author}} \\
	Lawrence Livermore National Laboratory \\
	Livermore, CA 94550 \\
	\texttt{reddy6@llnl.gov} \\
}

\date{}

\begin{document}

\maketitle

\begin{abstract}

    This work introduces structure preserving hierarchical decompositions for sampling Gaussian random fields (GRF) within the context of multilevel Bayesian inference in high-dimensional space. Existing scalable hierarchical sampling methods, such as those based on stochastic partial differential equation (SPDE), often reduce the dimensionality of the sample space at the cost of accuracy of inference. Other approaches, such that those based on Karhunen-Loève (KL) expansions, offer sample space dimensionality reduction but sacrifice GRF representation accuracy and ergodicity of the Markov Chain Monte Carlo (MCMC) sampler, and are computationally expensive for high-dimensional problems. The proposed method integrates the dimensionality reduction capabilities of KL expansions with the scalability of stochastic partial differential equation (SPDE)-based sampling, thereby providing a robust, unified framework for high-dimensional uncertainty quantification (UQ) that is scalable, accurate, preserves ergodicity, and offers dimensionality reduction of the sample space. The hierarchy in our multilevel algorithm is derived from the geometric multigrid hierarchy. By constructing a hierarchical decomposition that maintains the covariance structure across the levels in the hierarchy, the approach enables efficient coarse-to-fine sampling while ensuring that all samples are drawn from the desired distribution.    The effectiveness of the proposed method is demonstrated on a benchmark subsurface flow problem, demonstrating its effectiveness in improving computational efficiency and statistical accuracy. Our proposed technique is more efficient, accurate, and displays better convergence properties than existing methods for high-dimensional Bayesian inference problems.

\end{abstract}

\section{Introduction} \label{sec:intro}

Gaussian random fields (GRFs) play a pivotal role in modeling and understanding spatially correlated phenomena across various domains, including geostatistics \cite{Song2008}, physics \cite{Pen1997}, and machine learning \cite{Wang2023}. They provide a robust framework for encoding structural information about spatial fields, including smoothness and periodicity, through covariance functions or kernels. Markov chain Monte Carlo (MCMC) \cite{Robert2021} methods frequently intersect with GRFs in Bayesian inference and uncertainty quantification (UQ)where GRFs are often employed as priors and proposal distributions for spatial fields in inverse problems \cite{Cavieres2024}. The integration of GRFs and MCMC allows for the systematic incorporation of observational data and prior knowledge into statistical models, but it also introduces computational bottlenecks. These bottlenecks are particularly pronounced in high-dimensional problems, where the cost of generating GRF samples and evaluating forward models is substantial.

Efficient sampling of GRFs is a cornerstone of their utility, especially in high-dimensional settings where computational demands are significant. Traditional sampling methods, such as the Cholesky decomposition \cite{Yang2022} and Karhunen-Loève (KL) expansions \cite{Zheng2017}, are well-suited for lower dimensional problems. For instance, the KL expansion requires the eigendecomposition of a dense covariance matrix, which can be tractable for lower dimensional sampling but often infeasible for high-dimensional, large-scale problems, thereby necessitating the need for approximations - such as low-rank, randomized matrix \cite{Sabelfeld2011}, hierarchical matrix \cite{Khoromskij2009}, and Rayleigh-Ritz \cite{Ruiz1999} methods. To alleviate these challenges, scalable alternatives such as circulant embedding \cite{Graham2018} and stochastic partial differential equation (SPDE)-based sampling \cite{Lindgren11,Duswald2024} have often been preferred for large scale problems. The SPDE approach, in particular, exploits significant advances in scalable, numerical techniques for solving PDEs, such as higher order finite elements, preconditioning strategies, scalable linear algebra, with the GRF sample being realized as the solution to the SPDE. While surrogate models, in the form of Gaussian processes \cite{Paun2019,Smith2024}, polynomial chaos expansions \cite{Luthen2021,Liu2020}, and deep neural networks \cite{Yan2020,Hortua2020}, have greatly reduced the computation cost of the forward problem \cite{Robert2017}, they introduce approximation errors that bias the sampling. This, combined with the computational expense of Bayesian inference in high-dimensional sample space, makes inverse problems particularly challenging when using MCMC techniques.

Recent advancements in hierarchical GRF sampling have emerged as promising solutions to address these challenges, offering a balance between computational efficiency and accuracy. These hierarchical methods decompose the sampling process across multiple levels of resolution, enabling a coarse-to-fine representation of the field where the sampling on the less expensive, coarser resolutions is used to inform the sampling on the more expensive and accurate, finer resolution, improving the efficiency of the MCMC sampler and accelerating the convergence of the statistical moment estimation. The hierarchical MCMC sampling techniques have been extended and combined with variance-decay techniques, such as multilevel Monte Carlo (MLMC) sampling \cite{Dodwell2015,Dodwell2019}, to accelerate estimation of statistical moments using multilevel samples. Naturally, such methods require a hierarchical decomposition of the GRF that is consistent across all levels, to ensure that the samples on each level are drawn from the correct distribution (i.e. the decomposition preserves the structure of the covariance). Such a hierarchical decomposition based on multigrid theory has recently been developed for SPDE-based sampling, and has shown promise in improving the acceptance rate across the levels and the multilevel estimate of the statistical moments \cite{Fairbanks2021}. 

Hierarchical methods based solely on multigrid decomposition for SPDE sampling face challenges in balancing sample space reduction with discretization accuracy. Coarsening the resolution (i.e. mesh) reduces computational costs but may introduce errors that degrade the quality of GRF realizations and forward model evaluations. Hierarchical sampling via KL expansions offers a potential solution and can achieve high discretization accuracy by utilizing a finer (nodal) resolution (i.e. mesh) for the forward problem, while reducing dimensionality of the sample space by coarsening in modal space. However, hierarchical sampling based on the KL expansion has been limited to covariance kernels - such as exponential covariance \cite{LeMaitre2010,Teckentrup13,Lykkegaard2023,lykkegaard2021accelerating,Cliffe2011} - for which the analytical form of the KL expansion is known. This not only circumvents the computationally expensive eigendecomposition of the covariance matrix but also ensures the consistent ordering of the KL modes across all levels. Then, the sample space (i.e. number of terms in the expansion) can be held constant across all levels \cite{Cliffe2011} or varied on each level \cite{Teckentrup13,Lykkegaard2023,lykkegaard2021accelerating} according to appropriate error estimations for the truncation \cite{Charrier2012}. While discrete approximations of the KL modes can be computed for any arbitrary discrete covariance matrix, coupling KL sampling across different levels in a hierarchy requires additional work since the coefficients of the modes on one level do not necessarily correspond to those of the modes on another level. Furthermore, ensuring ergodicity of the sampling on each level requires the set of KL modes to be complete, which is intractable to compute and store in memory for fine resolutions. Hence, a hierarchical approach that performs dimensionality reduction of the sample space (i.e. use a truncated KL expansion) without sacrificing the discretization accuracy or ergodicity of the sampler is highly sought after.

We address this challenge and develop a hierarchical modal-multigrid decomposition that couples the discretization accuracy and dimensionality reduction of the KL sampling with the scalability and ergodicity-preservation of the SPDE sampler. We derive decompositions that preserve the structure of the covariance matrix and ensure the samples on each level are drawn from the correct distribution. The decomposition enables hierarchical sampling in the space complement to the space spanned by the truncated KL modes, thereby ensuring erogodicity and convergence to the correct posterior distribution. We employ our hierarchical decompositions within the context of multilevel Bayesian inference, namely, multilevel MCMC (MLMCMC), to estimate the posterior mean of a scalar quantity of interest (QoI) in a high-dimensional setting. The resulting MLMCMC sampler is demonstrated on a problem in subsurface flow, namely Darcy flow with uncertain permeability coefficient, whereby Darcy equations are solved within the mixed finite element framework. 

This paper is structured as follows: Section 2 provides an overview of GRF sampling using the KL expansion and SPDE approach, and establishes their equivalence. Section 3 introduces the hierarchical modal-multigrid decomposition, the coupling of the KL and SPDE samplers, and proves the multilevel decomposition samples from the desired covariance. Section 4 introduces the multilevel sampling via MLMCMC algorithm and the multilevel estimation of stochastic moment. Section 5 introduces the test problem, a benchmark groundwater flow problem, and its finite element discretization, and Section 6 demonstrates the performance of the hierarchical samplers. We conclude by summarizing the findings and outlining potential directions for future research in this dynamic field.

\section{Sampling Gaussian Random Fields} \label{sec:grfs}

The standard approach for generating realizations of Gaussian random fields with covariance function $\cov$ (i.e. kernel) can be viewed as a transformation of a standard multivariate normal distribution
\aligneq{StdSampling}{
    \MC{N}(0,\cov) \sim \theta &= \cov^{1/2} \xi~, \quad~\mathrm{with~} \xi \sim \MC{N}(0,I)
}
and requires computing the square-root of the covariance function. Defining $\cov^{1/2}$ within different frameworks (e.g., by Cholesky factors, Fourier expansion, etc.) yields different sampling algorithms with different algorithmic scaling and efficiency. However, in practice, this square-root cannot be computed analytically for most covariance kernels and numerical approaches are sought after. Here, we focus on two popular and optimal approaches for sampling: the \KL expansion and the stochastic partial differential equation.

\subsection{Sampling Technique: \KL Expansion} \label{subsec:samptech:KLE}

The \KL (KL) expansion represents a stochastic process, such as a Gaussian random field, as an infinite series of orthogonal (eigenfunctions) functions weighted by random coefficients. It is known by Mercer's theorem, that any symmetric, positive definite (SPD) kernel $\cov$ admits a spectral decomposition
\eq{cov:kernel-decomp}{
    \cov (\Vx,\Vy) = \sum_{k=1}^\infty \lambda_k \psi_k(\Vx) \psi_k(\Vy),
}
with eigenvalues, $\lambda_k$, and eigenfunctions, $\psi_k(\Vx)$, where 
\[
    \integral{\Omega}{}{\psi_i(\Vx) \psi_j(\Vx)~d\Omega} = \delta_{ij} .
\]    
Then, a stochastic process with mean $\mu$ and covariance $\cov$ can be represented in the eigenfunctions of $\cov$
\[
    \MC{N}(\mu,\cov) \sim \theta(\Vx)  =  \mu(\Vx)+ \sum_{i=1}^\infty \zeta_i \psi_i(\Vx),
\]
where the random coefficients $\zeta_i$ can be found by projecting $\theta(\Vx)-\mu(\Vx)$ onto the eigenfunctions
\[
    \zeta_i = \integral{\Omega}{}{ \rbrac{\theta(\Vx)-\mu(\Vx)} \psi_i(\Vx) ~d\Omega}.
\]
Here onwards, we assume $\mu(\Vx) = 0$ without loss of generality. Then, we obtain the following moments for $\zeta_i$
\eq{KL:mean}{
    \E{\zeta_i} = \E{\integral{\Omega}{}{ \theta(\Vx) \psi_i(\Vx) ~d\Omega}} = \integral{\Omega}{}{ \E{\theta(\Vx)} \psi_i(\Vx) ~d\Omega} = 0,
}
and
\aligneq{KL:cov}{
    \E{\zeta_i \zeta_j} &= \E{\integral{\Omega}{}{ \integral{\Omega}{}{ \theta(\Vx)\theta(\Vy) \psi_i(\Vx)\psi_j(\Vy) ~d\Vx}~d\Vy}} \\
                        &= \integral{\Omega}{}{ \integral{\Omega}{}{ \E{\theta(\Vx)\theta(\Vy)} \psi_i(\Vx)\psi_j(\Vy) ~d\Vx}~d\Vy} \\
                        &= \integral{\Omega}{}{ \psi_j(\Vy) \rbrac{\integral{\Omega}{}{ \cov(\Vx,\Vy) \psi_i(\Vx) ~d\Vx}}~d\Vy} \\
                        &= \lambda_i \delta_{ij}, \\
}
where, from \eqref{cov:kernel-decomp}, we use the relation 
\aligneq{KL:kernel-integral}{
    \integral{\Omega}{}{ \cov(\Vx,\Vy) \psi_i(\Vx) ~d\Vx} &= \sum_{k=1}^\infty \lambda_k \psi_k(\Vy) \integral{\Omega}{}{  \psi_k(\Vx) \psi_i(\Vx) ~d\Vx}\\
                                                            &= \sum_{k=1}^\infty \lambda_k \psi_k(\Vy) \delta_{ik} = \lambda_i \psi_i(\Vy).
}
Then, with $\zeta_i \sim \MC{N}(0,\lambda_i)$ we get
\aligneq{KL:continuous}{
    \MC{N}(0,\cov) \sim \theta(\Vx)  = \sum_{i=1}^\infty \underbrace{\xi_i \sqrt{\lambda_i}}_{\zeta_i }  \psi_i(\Vx), \quad \mathrm{with}~\xi_i \sim \MC{N}(0,1),
}
where, given $\lambda_i$ and $\psi_i(\Vx)$, sampling $\MC{N}(0,\cov)$ now requires sampling a standard multivariate normal distribution $\MC{N}(0,I)$. Since analytical representations of eigenfunctions are only known for select few Mercer kernels, the eigenfunctions for arbitrary covariance kernels are often approximated discretely and computed numerically. Similarly, an SPD covariance matrix $\cov_h \in \R^{N\times N}$, induced by the covariance kernel, $\cov$, admits a spectral decomposition
\eq{cov:spec-decomp}{
    \cov_h = \sum_{i=1}^N  \lambda_i \psih_i~\psih_i^T \quad \rightarrow \quad \cov_h^{1/2} = \sum_{i=1}^N  \sqrt{\lambda_i} \psih_i~\psih_i^T,
}
with eigenvectors $\psih_i$. We can then compute \eqref{StdSampling} as
\aligneq{KL:discrete}{
    \MC{N}(0,\cov_h) \sim \theta_h &= \sum_{i=1}^N  \underbrace{\sqrt{\lambda_i}(\psih_i,\xi)}_{\zeta_i}  \psih_i \quad ~&\mathrm{with~} \zeta_i \sim \MC{N}(0,\lambda_i \MC{M}_i)\\
                          &= \sum_{i=1}^N  \xi_i \mathpzc{m}_i \sqrt{\lambda_i} \psih_i \quad ~&\mathrm{with~} \xi_i \sim \MC{N}(0,1) &\mathrm{,~and~~} \mathpzc{m}_i = (\psih_i,\psih_i)_{\MC{M}^{1/2}}, \\
}
where, the inner product $(\psih_i,\xi) \sim \MC{N}(0,\MC{M}_i)$ and $\MC{M} = (\cdot,\cdot)$ is the bilinear form on the space of $\psih$ and $\xi$.

\subsection{Sampling Technique: Stochastic PDE Approach} \label{subsec:samptech:SPDE}

The stochastic partial differential equation (SPDE) approach defines $\cov^{1/2}$ through a differential operator. Consider a trace class operator, $\cov = \MC{A}^{-2}$ with
\eq{cov:operator}{
    \MC{A}\theta(\Vx) := \dfrac{1}{g} \rbrac{ \Lap{} - \kappa^2 }^{\alpha/2} \theta(\Vx) ,~\mathrm{with}~ \alpha = \nu + \dfrac{d}{2},~\kappa > 0,~ \nu > 0
}
where $\kappa^{-1} > 0$ is the correlation length, 
\eq{spde:g}{
    g = (4\pi)^{d/4} \kappa^\nu \sqrt{\dfrac{\Gamma(\nu + d/2)}{\Gamma(\nu)}}
}
imposes a unit marginal variance on $\theta(\Vx)$ and $d$ is the dimension. In particular, $\MC{A}^{-2}$ with \eqref{cov:operator} represents a class of Mat\'ern covariance kernels 
\eq{matern}{
    \cov_M(\Vx,\Vy) \propto \dfrac{\sigma^2}{2^{\nu - 1}\Gamma(\nu)} \rbrac{\kappa \Norm{\Vx - \Vy}}^\nu K_\nu\rbrac{\kappa \Norm{\Vx - \Vy}},
}
where, $\sigma^2$ is the marginal variance, $\nu > 0$ determines the mean-square differentiability of the stochastic process, $\Gamma(\nu)$ the gamma function and $K_\nu$ is the modified Bessel function of the second kind. In the special case of $\nu = \dfrac{1}{2}$, \eqref{matern} reduces to the exponential covariance
\eq{expcov}{
    \cov_E(\Vx,\Vy) \propto \sigma^2 \exp(\kappa \Norm{\Vx - \Vy}),
}
Then, for $d=2,~\nu=1$ and $d=3,~\nu=\dfrac{1}{2}$, the fractional-order differential operator in \eqref{cov:operator} reduces to an integer-order operator. Following the definition of $\MC{A}$, the random fields can be sampled as
\aligneq{cov:continuous}{
    \MC{N}(0,\cov) \sim \theta(\Vx) &= \cov^{1/2} \MC{W}(\Vx)%
                          &= \MC{A}^{-1} \MC{W}(\Vx) \rightarrow  \MC{A}\theta(\Vx) = \MC{W}(\Vx).
}
Here, $\cov^{1/2} = \MC{A}^{-1}$, hence, \eqref{StdSampling}, \eqref{KL:continuous}, and \eqref{cov:continuous} are all equivalent if $\psi_i$ are the eigenfunctions of $\cov$ and $\MC{A}$. Like in the case of analytical \KL expansion, analytic solutions to \eqref{cov:continuous} can only be obtained in special cases on simple domains, and therefore, require solution by numerical methods. In this work, we employ the finite element method to solve \eqref{cov:continuous} with $\MC{A}$ defined by \eqref{cov:operator}.

Consider trial and test spaces $\varphi_h,\phi_h \in H^1(\Omega_h) \subset H^1(\Omega)$, respectively, the finite element representation of the GRF $\theta_h \in H^1(\Omega_h)$ expanded in basis $\varphi_h$.

\begin{problem}
    Find $\theta_h \in H^1(\Omega_h)$ such that
    \aligneq{SPDE:H1}{
            \rbrac{\Grad{\phi_h},\Grad{\theta_h}} - \kappa^2 \rbrac{\phi_h,\theta_h} &= -g \rbrac{\phi_h,\MC{W}_h}, \quad &\forall \phi_h \in H^1(\Omega_h)
    }
    with boundary conditions $\Grad{\theta_h} \cdot \Vec{n} = 0$,
and inner products defined as 
\expression{
    \rbrac{\phi_h,\theta_h} = \integral{\Omega_h}{}{\phi_h~\theta_h~d\Omega},~\mathrm{and} ~\rbrac{\Grad{\phi_h},\Grad{\theta_h}} = \integral{\Omega_h}{}{\Grad{\phi_h}\cdot\Grad{\theta_h}~d\Omega}.
}
\end{problem}
To simplify the presentation, we write \eqref{SPDE:H1} as 
\aligneq{cov:discrete}{
            (\MC{S}_h - \kappa^2 \MC{M}_h)\theta_h &= -g \MC{M}_h \MC{W}_h \\
            \MC{A}_h \theta_h &= -g \zeta_h, %
}
with mass matrix $\MC{M}_h = \rbrac{\varphi_h,\phi_h}$, and stiffness matrix $ \MC{S}_h = \rbrac{\Grad{\phi_h},\Grad{\varphi_h}}$. Here, \eqref{cov:discrete} is the discrete form of \eqref{cov:continuous}. From \eqref{KL:mean}, \eqref{KL:cov} and \eqref{KL:kernel-integral}, and by representing the eigenfunctions $\psi(\Vx)$ in \eqref{KL:cov} in the finite element basis $\varphi_h$, we get
\[
   \E{\zeta_h \zeta_h} = \integral{\Omega_h}{}{ \varphi_h ~\varphi_h  ~d\Omega} = (\varphi_h,\varphi_h)= \MC{M}_h . %
\]
Hence, $\zeta_h \sim \MC{N}(0,\MC{M}_h)$ and 
\[
    \MC{N}(0,\MC{M}_h) \sim \zeta_h \sim \MC{M}_h\MC{W}_h \iff \MC{W}_h \sim \MC{N}(0,\MC{M}_h^{-1}).
\]
Defining $\zeta_h = \MC{M}_h^{1/2} \xi_h$ with $\xi_h \sim \MC{N}(0,I)$, we can generate realizations of GRFs by solving
\aligneq{cov:discrete:std}{
    \MC{A}_h \theta_h = -g\MC{M}_h^{1/2} \xi_h  \rightarrow  \theta_h = -g \MC{A}_h^{-1} \MC{M}_h^{1/2} \xi_h. %
}
It should be noted that the bilinear form $\MC{A}_h$ may not have the same eigenvectors as the covariance matrix induced by its kernel (\eqref{matern} and \eqref{expcov}). When $\psih$ are the eigenvectors of $\MC{A}_h$, \eqref{cov:discrete:std} and \eqref{KL:discrete} are equivalent. Here onwards, we assume $(\lambda_i,\psih_i)$ are the eigenvalue-eigenvector pairs of $\MC{A}^{-2}_h$, and are known and ordered.

\section{Multilevel Decomposition for Hierarchical Sampling} \label{sec:mldecomposition}

Solutions to \eqref{cov:discrete:std} with independently sampled $\xi_h$ yield realizations of GRF that are independent of the other levels in the hierarchy. For multilevel sampling, we require the fine-level fields to be conditioned on the coarse-level fields. Consider a hierarchy of nested spaces $\Theta_0 \subset \Theta_1 \ldots \subset \Theta_\ell \subset \Theta_{\ell+1} \ldots \subset \Theta_N \in H^1(\Omega_h)$. Without loss of generality, denote by subscripts $\ell$ and $L=\ell+1$, the coarse, $\Theta_\ell$, and fine, $\Theta_L$, spaces, respectively. Furthermore, denote by superscripts the realization of a field $\eta$ on one level on another (i.e. $\eta_\ell^L \in \Theta_l$ represents the coarse level realization of $\eta_\ell \in \Theta_\ell$ in the basis of the fine level $L$) with $\eta_\ell^L = \MC{I}\eta_\ell$ where $\MC{I}:\Theta_\ell \rightarrow \Theta_L$ is the interpolation/prolongation operator and $\MC{R}: \Theta_L \rightarrow \Theta_\ell$ is the restriction operator such that $\MC{R}\MC{I} = \mathrm{Id}$ in $\Theta_\ell$; it should be mentioned that the matrix representation of $\MC{R}\MC{I}$ is not identity since different bases are used for $\Theta_\ell$ and $\Theta_L$. In this section, we formulate multilevel decompositions of the field $\eta_h$ that enable sampling of realizations across different levels as well as sampling in the complement space $\Theta_L \backslash \Theta_\ell$. Furthermore, the decompositions will enable coupling between the (modal) \KL and (nodal) SPDE samplers.
To generate a random field on the fine level $\widetilde{\eta}_L$ conditioned on the coarse level $\eta_\ell^L$, we employ the following decomposition of $\eta$
\aligneq{mldecomp}{
    \widetilde{\eta}_L &= \underbrace{\MC{Q}_L \eta_\ell^L }_{\in \Theta_\ell} + \underbrace{(I - \MC{Q}_L) \eta_L}_{\in \Theta_L \textbackslash \Theta_\ell},
}
where the $L^2$ projection $\MC{Q}_L : \Theta_L \rightarrow \Theta_\ell$ will be constructed according to the samplers on the coarse and fine level. Throughout the remainder of the manuscript, we use the notation convention $f_\ell \simeq g_\ell$ to mean that there exist positive constants $c_1$ and $c_2$, independent of $\ell$, such that $c_1 f_\ell \leq g_\ell \leq c_2 f_\ell$ for all non-negative scalars $f_\ell$ and $g_\ell$. Additionally, we make the following assumptions:

\begin{assumption}\normalfont \label{assume:iid}
$\eta_\ell^L$ and $\eta_L$ are independent, hence, $\E{\eta_\ell^L (\eta_L)^T} = \E{\eta_\ell^L} \E{(\eta_L)^T} = 0$
\end{assumption}
\begin{assumption}\normalfont \label{assume:specequiv}
    Let the coarse level covariance, $\cov_\ell$, and the fine level covariance, $\cov_L$, be defined by bilinear forms on $\Theta_\ell$ and $\Theta_L$, respectively. Then, we assume that $\cov_\ell$ and $\cov_L$ are spectrally equivalent. That is, for $\eta_\ell \in \Theta_\ell$ and $\eta^L_\ell \in \Theta_L$, we have
    \expression{
        \Norm{\eta_\ell}_{\cov_\ell} \simeq \Norm{\eta^L_\ell}_{\cov_L} ~\forall \eta_\ell,
    }
    Furthermore, we have $\cov_\ell = \MC{I}^T \cov_L \MC{I}$ and 
    \expression{
        (\eta^L_\ell)^T \cov_L \eta^L_\ell = \eta^T_\ell \MC{I}^T \cov_L \MC{I} \eta_\ell =      \eta^T_\ell \cov_\ell \eta_\ell ~~\forall \eta_\ell.    
    }

\end{assumption}

The covariance of the coarse realization in fine basis also satisfies spectral equivalence and is given by the following lemma.
\begin{lemma}\normalfont \label{lemma:coarse_cov}
    Let $\eta_\ell \sim \MC{N}(0,K_\ell)$ be a random function on level $\ell$ with covariance $K_\ell$, and $\eta_\ell^L = \MC{I}\eta_\ell$ is the realization of $\eta_\ell$ in the basis of level $L$ with $\MC{R} = \MC{I}^T$.
    Then, $\E{\MC{Q}_L\eta_\ell^L~(\MC{Q}_L\eta_\ell^L)^T} = K^L_\ell \simeq K_\ell$.
\end{lemma}
\begin{proof}
We begin by noting $ \MC{Q}_L\eta_\ell^L = \eta_\ell^L$ and 
\expression{
\E{\eta_\ell^L~(\eta_\ell^L)^T} = \MC{I}\E{\eta_\ell~\eta^T_\ell} \MC{I}^T = \MC{I} K_\ell \MC{I}^T
}
Then,
\aligneq{lemme:coarse_cov}{
\eta^T_\ell K_\ell~\eta_\ell &\simeq  (\eta_\ell^L)^T K^L_\ell \eta_\ell^L  \\
&= (\eta_\ell^L)^T  \MC{I} K_\ell \MC{I}^T \eta_\ell^L \\
&= \eta^T_\ell \MC{I}^T  \MC{I} K_\ell \MC{I}^T \MC{I} \eta_\ell \\
&= \eta^T_\ell K_\ell \eta_\ell ,
}
where $\MC{R}\MC{I} = \MC{I}^T \MC{I} = \mathrm{Id}$. Hence, $c_1 = c_2 = 1$.
\end{proof}
   
Following Assumptions \ref{assume:iid} and \ref{assume:specequiv}, and Lemma \ref{lemma:coarse_cov} we present the main theorem of this work which will enable coupling between different sampling techniques across levels.
\begin{theorem}\normalfont \label{theorem:decomp}
Let $\eta_\ell \sim \MC{N}(0,K_\ell)$ and $\eta_L \sim \MC{N}(0,K_L)$ be random functions on levels $\ell$ and $L$ with covariances $K_\ell$ and $K_L$, respectively. Let $K_\ell$ and $K_L$ be spectrally equivalent. Then, the decomposition satisfies
\expression{
    \widetilde{\eta}_L = \MC{Q}_L \eta^L_\ell + (I - \MC{Q}_L) \eta_L \sim \MC{N}(0,K_L).
}
\end{theorem}
\begin{proof}
    The condition on expectation $\E{\widetilde{\eta}_L} = \MC{Q}_L \E{\eta^L_\ell} + (I - \MC{Q}) \E{\eta_L}=0$ follows trivially. To show that $\E{\widetilde{\eta}_L~(\widetilde{\eta}_L)^T} = K_L$, we compute
    \aligneq{thm:decomp:covariance}{
        \E{\widetilde{\eta}_L~(\widetilde{\eta}_L)^T} &= \E{\MC{Q}_L \eta^L_\ell (\MC{Q}_L \eta^L_\ell)^T} + \E{(I - \MC{Q}) \eta_L (I - \MC{Q}) \eta_L^T} \\
        &=  \E{\eta^L_\ell (\eta^L_\ell)^T} + (I - \MC{Q}) \E{\eta_L \eta_L^T} (I - \MC{Q}^T) \\
        &=  K^L_\ell  + (I - \MC{Q}) K_L (I - \MC{Q}^T) \\
        &=  K^L_\ell  + K_L - \MC{Q} K_L - K_L\MC{Q}^T + \MC{Q} K_L \MC{Q}^T \\
        &=  K_L, \\
    }
    where we have exploited the idempotency ($\MC{Q}^2 = \MC{Q}$) and the symmetry of the projection $\MC{Q} K_L = K_L \MC{Q}^T = \MC{Q} K_L \MC{Q}^T = K^L_\ell$.
\end{proof}

\subsection{Hierarchical Sampling: Multigrid Decomposition} \label{subsec:pdedecomp}

Here, we seek to define appropriate projections, $\MC{Q}$, for the PDE hierarchy that preserve the covariance on the coarse and fine levels. For the SPDE sampling, we apply appropriate transformations $\xi_h \sim \MC{N}(0,I)$ to compute the forcing term $\zeta_h = \MC{M}_h \MC{W}_h = \MC{M}^{1/2}_h \xi_h \sim \MC{N}(0,\MC{M}_h)$.

Consider a covariance $K_L$ defined by a symmetric, positive-definite (SPD) bilinear form $a(u_L,\phi_L)$ with $u_L,\phi_L \in \Theta_L$. Then, the Galerkin projection, $\pi_A: \Theta_L \rightarrow \Theta_\ell$ , of $u_L \in \Theta_L$ to $u_\ell \in \Theta_\ell$ that solves 
\expression{
    a(u_L-u_\ell,\phi_L) = 0, \quad \forall \phi_L \in \Theta_L
}
is defined as $\MC{Q}_L = P \Pi_A$ with $\Pi_A P = I$, where $P$ is the prolongation/interpolation operator that represents $u_\ell$ on the coarse level in the basis on the fine level $\varphi_L$ (i.e. $u^L_\ell$) and $\Pi_A = A^{-1}_\ell P^T A_L$ is the restriction operator with the discretized bilinear form, $A$. Using this construction, we can define the projection that preserves the covariance of the multilevel samples of $\eta$ on all levels.
As in \cite{Fairbanks2021}, we perform decomposition of $ \MC{W} \sim \MC{N}(0,\MC{M}^{-1})$ (and equivalently $ \zeta \sim \MC{N}(0,\MC{M})$) with \eqref{mldecomp}, and hence, defined the $L^2$ projection $\MC{Q}_L = P\Pi_{\MC{M}}$ with $\Pi_{\MC{M}} = \MC{M}^{-1}_\ell P^T \MC{M}_L$. Here, we can verify the idempotency $\MC{Q}^2_L = \MC{Q}_L$. Furthermore, using the relation $\MC{M}_\ell = P^T \MC{M}_L P$, we can show that $\MC{Q}_L \MC{M}^{-1}_L = \MC{M}^{-1}_L \MC{Q}_L^T = \MC{Q}_L \MC{M}^{-1}_L \MC{Q}_L^T = P \MC{M}^{-1}_\ell P^T$. 
From Lemma \ref{lemma:coarse_cov}, we have
\[
    P \MC{M}^{-1}_\ell P^T = (\MC{M}^L_\ell)^{-1} \simeq \MC{M}_\ell^{-1}.
\]
For convenience, here onwards, we ommit the subscript $\MC{M}$ from the restriction operator. Hence, unless otherwise stated, $\Pi = \Pi_\MC{M}$.

We define the fine realization of a coarse field as $P \Pi \MC{W}_\ell^L = P \MC{W}_\ell$ and a multilevel realization of the fine level standard white noise, $\widetilde{\MC{W}}_L$, conditioned on the coarse level noise, $\MC{W}_\ell$, as
\aligneq{hpde:decomp}{
    \widetilde{\MC{W}}_L &= P \Pi \MC{W}_\ell^L + (I - P\Pi) \MC{W}_L\\
                         &= P \MC{W}_\ell + (I - P \Pi) \MC{W}_L \\
                         &= \underbrace{P \MC{M}^{-1/2}_\ell \xi_\ell}_{\in \Theta_\ell} + \underbrace{(I - P \Pi) \MC{M}^{-1/2}_L \xi_L}_{\in \Theta_L \textbackslash \Theta_\ell}. \\
}
Since $\MC{W}_\ell$ and $\MC{W}_L$ are independent, from Theorem \ref{theorem:decomp} we have that $\widetilde{\MC{W}}_L \sim \MC{N}(0,\MC{M}^{-1}_L)$, hence $\widetilde{\zeta}_L = \MC{M}_L \widetilde{\MC{W}}_L \sim \MC{N}(0,\MC{M}_L)$. Then, samples of GRF on levels $\ell$ and $L$ can be computed as 
\aligneq{hpde:coarse-sample}{
   \MC{N}(0,\cov_\ell) \sim \theta_\ell = \MC{A}_\ell^{-1} \zeta_\ell = \MC{A}_\ell^{-1} \MC{M}_\ell \MC{W}_\ell = \MC{A}_\ell^{-1} \MC{M}_\ell^{1/2} \xi_\ell, \\
}
and
\aligneq{hpde:fine-sample}{
   \MC{N}(0,\cov_L) \sim \widetilde{\theta}_L &= \MC{A}_L^{-1} \widetilde{\zeta}_L \\
                            &= \MC{A}_L^{-1} \MC{M}_L \widetilde{\MC{W}}_L\\
                            &= \MC{A}_L^{-1} \rbrac{ \MC{M}_L P \MC{M}^{-1/2}_\ell \xi_\ell +  \MC{M}_L (I - P \Pi) \MC{M}^{-1/2}_L \xi_L}\\
                            &= \MC{A}_L^{-1} \rbrac{ \MC{M}_L P \MC{M}^{-1/2}_\ell \xi_\ell +  (I - \Pi^T P^T) \MC{M}^{1/2}_L \xi_L} \\
                            &= \MC{A}_L^{-1} \big( \underbrace{ \Pi^T \zeta_\ell}_{ \in \Theta_\ell} + \underbrace{(I - \Pi^T P^T) \zeta_L}_{ \in \Theta_L \textbackslash \Theta_\ell } \big) . \\
}

Figure \ref{fig:spde:multilevelgrfs} shows a multilevel realization of a Gaussian random field sampled using the SPDE approach. It can be seen that the coarse level captures a significant portion of the structure of the realization, while the fine level captures the high-frequency components in the complement space, $\Theta_L \backslash \Theta_\ell$, that are not captured by the coarse level.

\begin{figure}[H] \centering
\begin{subfigure}[h]{0.33\textwidth} %
\includegraphicsifexists[trim={4cm 4cm 4cm 4cm},clip,width=\textwidth]{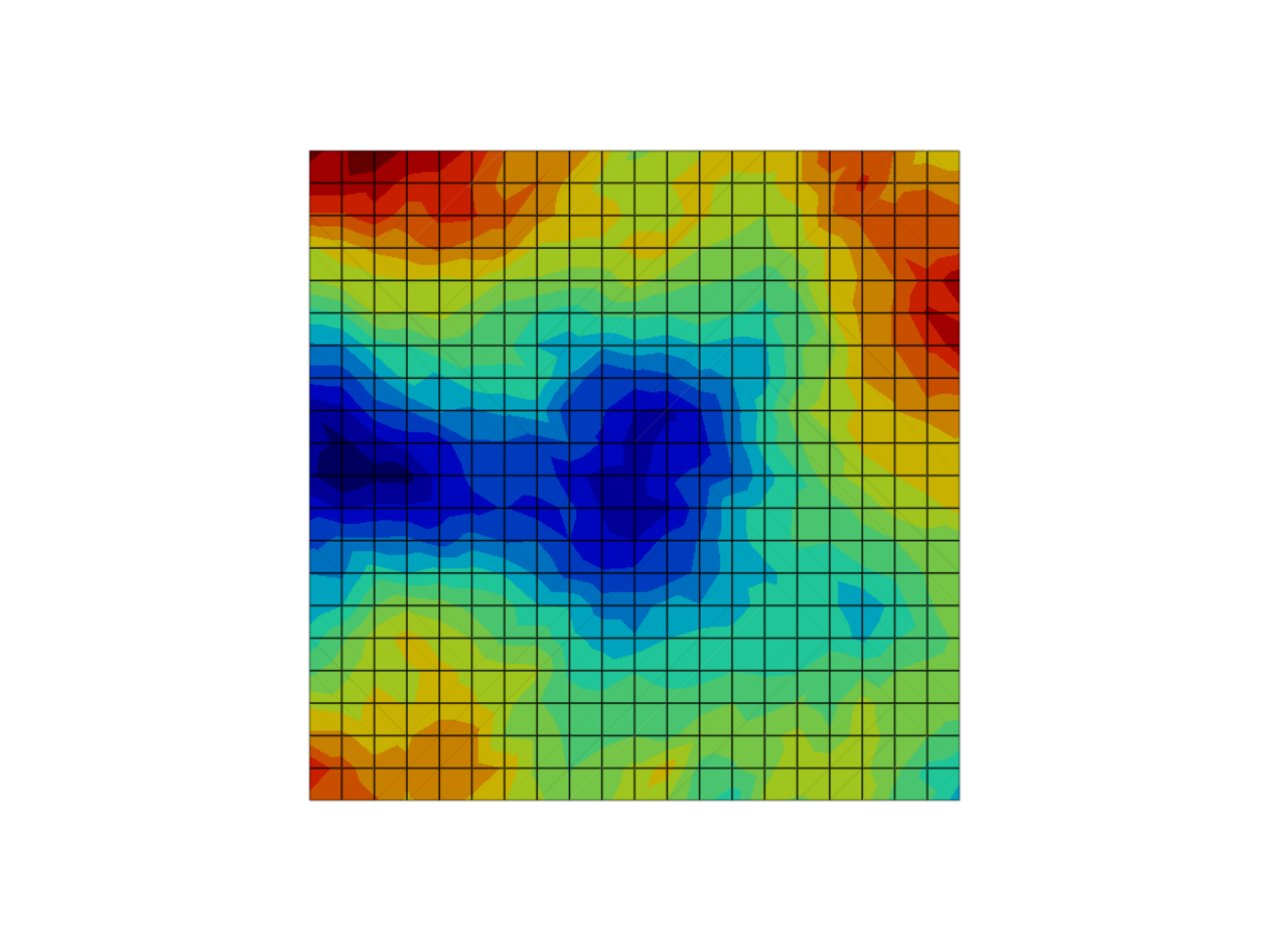} \caption{$\theta_0 \in \Theta_0$}\label{fig:spde:l0}
\end{subfigure}
\begin{subfigure}[h]{0.33\textwidth} %
\includegraphicsifexists[trim={4cm 4cm 4cm 4cm},clip,width=\textwidth]{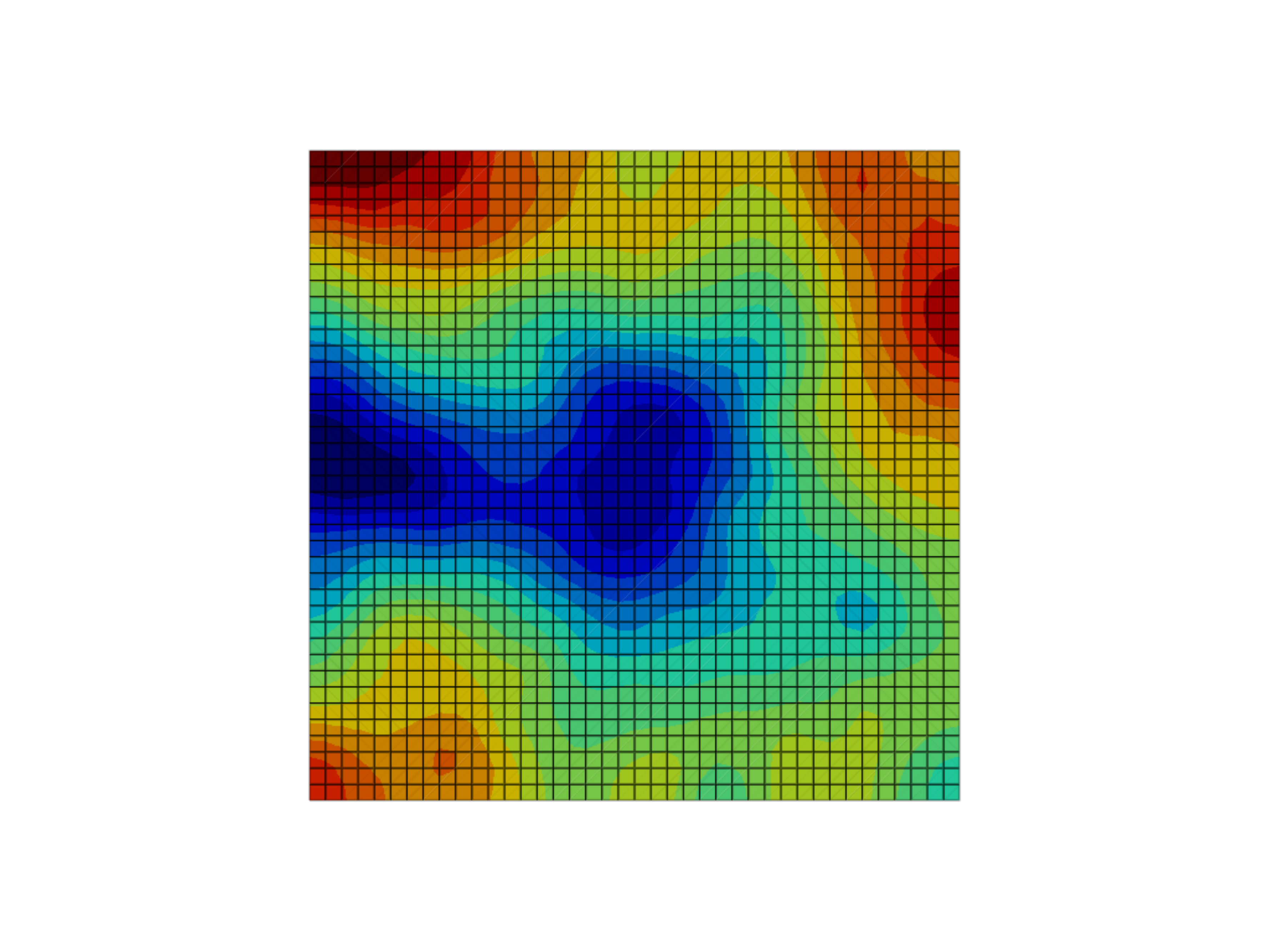} \caption{$\theta^1_0 \in \Theta_0$}\label{fig:spde:prolongate:l0}
\end{subfigure}
\begin{subfigure}[h]{0.33\textwidth} %
\includegraphicsifexists[trim={4cm 4cm 4cm 4cm},clip,width=\textwidth]{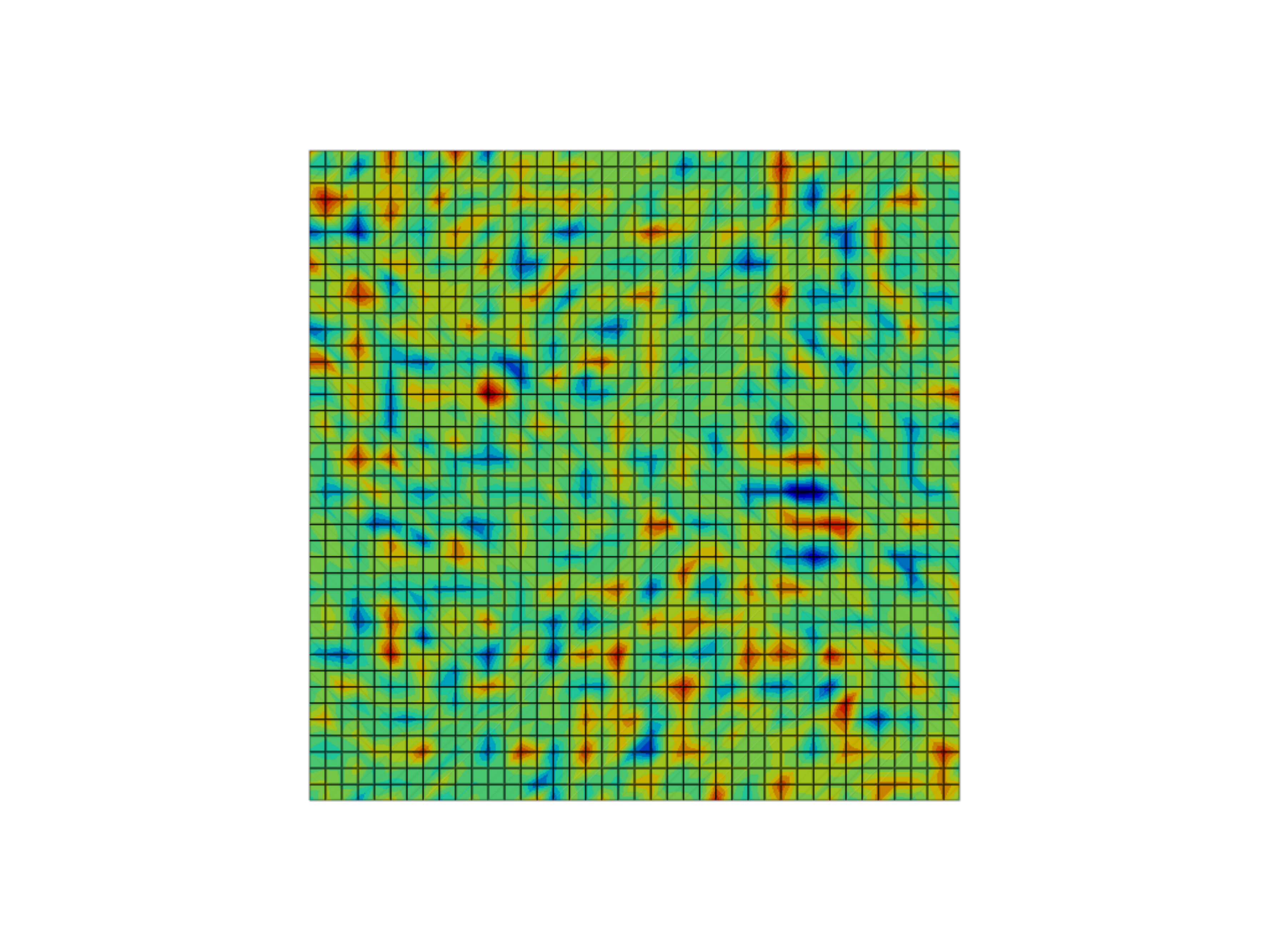} \caption{$\theta^\perp_1 \in \Theta_1 \backslash \Theta_0$}\label{fig:spde:complement:l1}
\end{subfigure}
\begin{subfigure}[h]{0.33\textwidth} %
\includegraphicsifexists[trim={4cm 4cm 4cm 4cm},clip,width=\textwidth]{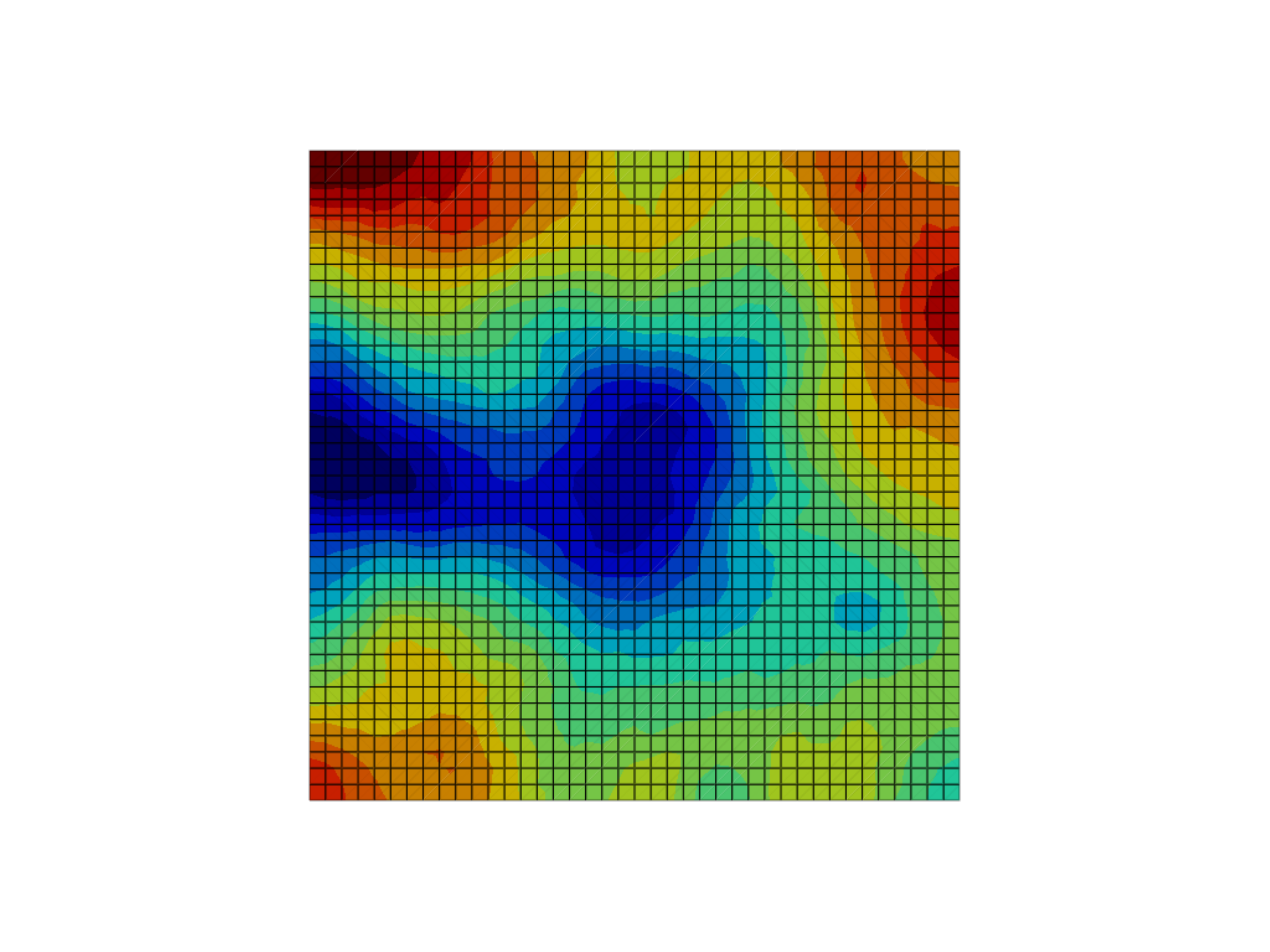} \caption{$\widetilde{\theta}_1 = \theta^1_0 + \theta^\perp_1 \in \Theta_1$}\label{fig:spde:complete:l1}
\end{subfigure}
\captionsetup{singlelinecheck=off,font=footnotesize}
\caption[]{A realization of a Gaussian random field across the levels of a two-level hierarchy sampled using the SPDE approach with: a) a sample on the coarse level, b) the coarse sample on the fine-level, c) fine-level sample in the coarse complement, and d) the complete multilevel realization of the GRF.}
\label{fig:spde:multilevelgrfs}
\end{figure}

\subsection{Hierarchical Sampling: Coupled \KL-Multigrid Decomposition} \label{subsec:klpdedecomp}

Here, we construct a projection that enables coupling between the KL and SPDE samplers. We consider an $m-$truncated KL expansion (\eqref{KL:discrete}) on the coarse level $\ell$ and an SPDE sampler on the fine level $L$. Since we use a truncated set of eigenvectors $\Psih = [\psih_1,\ldots,\psih_m] \in \R^{N \times m}$, we have the following space inclusion $\Thetah \subset \Theta_\ell \subset \Theta_L$ with $\Thetah = \mathrm{span}(\Psih)$. Applying the decomposition in \eqref{hpde:decomp} yields $\widetilde{\eta}_L \in \Thetah \cup \Theta_L \backslash \Theta_\ell$ but not in $\Theta_\ell \backslash \Thetah$, and therefore results in multilevel MCMC sampling that is not ergodic in $\Theta_L$. Note that when $\Psih$ is complete, $\Thetah = \Theta_\ell$ and ergodicity in $\Theta_L$ is restored. 

Hence, for complete sampling of $\Theta_L$ we define the projection $\MC{Q}_L = P \WH{\MC{Q}} \Pi$ which yields the following decomposition
\aligneq{klspde:decomp:almost-complete}{
    \widetilde{\MC{W}}_L &= \underbrace{P \WH{\MC{Q}} \Pi_{\MC{M}} \WH{\MC{W}}^L_\ell}_{\in \Thetah} + \underbrace{(I - P \WH{\MC{Q}} \Pi_{\MC{M}}) \MC{W}_L}_{\in \Theta_L \textbackslash \Thetah}, %
}
with the $L^2$ projection $\WH{\MC{Q}} =  \Psih \Psih^T : \Theta_\ell \rightarrow \Thetah$, where $\MC{W}_\ell = \Pi \MC{W}^L_\ell \in \Theta_\ell$, $\WH{\MC{W}}_\ell = \WH{\MC{Q}} \MC{W}_\ell \in  \Thetah$. It is easy to verify that both $\WH{\MC{Q}}$ and $P\WH{\MC{Q}} \Pi$ are idempotent. Using the idempotency of $\WH{\MC{Q}}$ and the commutative property of $\WH{\MC{Q}}$ and $\MC{M}_\ell$ (i.e. $\WH{\MC{Q}}\MC{M}_\ell = \MC{M}_\ell\WH{\MC{Q}}$), we can show that
\[
    P \WH{\MC{Q}} \Pi \MC{M}^{-1}_L = \MC{M}^{-1}_L \Pi^T \WH{\MC{Q}} P^T = P \WH{\MC{Q}} \Pi \MC{M}^{-1}_L \Pi^T \WH{\MC{Q}} P^T = P \MC{M}^{-1}_\ell \Psih \Psih^T P^T = (\WH{\MC{M}}^L_\ell)^{-1} \simeq \WH{\MC{M}}_\ell^{-1} \simeq \MC{M}_\ell^{-1},
\]
where 
\begin{align*}
    P \MC{M}^{-1}_\ell \Psih \Psih^T P^T &= P \Psih \Psih^T \MC{M}^{-1}_\ell \Psih \Psih^T P^T  \\
    & = (\Psih^L) \Psih^T \MC{M}^{-1}_\ell \Psih (\Psih^L)^T \\
    & = (\Psih^L) \WH{\MC{M}}^{-1}_\ell (\Psih^L)^T = (\WH{\MC{M}}^L_\ell)^{-1} \simeq \WH{\MC{M}}_\ell^{-1} \simeq \MC{M}_\ell^{-1}, \\
\end{align*}
$\WH{\MC{M}}^{-1}_\ell = \Psih^T \MC{M}^{-1}_\ell \Psih$ 
, and $\Psih^L = P \Psih$ represent the coarse-level KL-basis on the fine level.
Furthermore, we have 
\begin{align*}
    \WH{\MC{W}}_\ell = \MC{M}_\ell^{-1/2} \WH{\xi}_\ell &= \MC{M}_\ell^{-1/2} \WH{\MC{Q}} \xi_\ell = \MC{M}_\ell^{-1/2} \Psih (\Psih,\xi_\ell),
\end{align*}
where $(\psih_i,\xi_\ell) \sim \MC{N}(0,1)$ are the coefficients sampled in the KL sampler. Then, the complete decomposition that ensures the MCMC sampler is ergodic in $\Theta_L$, is given by
\aligneq{klspde:decomp:complete}{
    \widetilde{\MC{W}}_L &= \underbrace{P \MC{M}_\ell^{-1/2} \Psih \Vec{\WH{\xi}}}_{\in \Thetah} + \underbrace{(I - P \Psih \Psih^T \Pi) \MC{W}_L}_{\in \Theta_L \textbackslash \Thetah} ,\\    
}
where, $\Vec{\WH{\xi}} \in \R^m$ is a vector of the sampled KL coefficients. Noting that $\MC{M}_\ell$ is full rank and $\Psih$ is not complete, we get 
\[
    \WH{\MC{W}}_\ell = \WH{\MC{Q}} \MC{M}_\ell^{-1/2} \WH{\MC{Q}} \xi_\ell = \Psih \Psih^T \MC{M}_\ell^{-1/2} \Psih (\Psih,\xi_\ell) = \Psih \WH{\MC{M}}_\ell^{-1/2} (\Psih,\xi_\ell),
\]
and an alternative form of the multilevel decomposition 
\aligneq{klspde:decomp:alternative}{
    \widetilde{\MC{W}}_L &= \underbrace{P \Psih \WH{\MC{M}}_\ell^{-1/2} \Vec{\WH{\xi}}}_{\in \Thetah} + \underbrace{(I - P \Psih \Psih^T \Pi) \MC{W}_L}_{\in \Theta_L \textbackslash \Thetah}, \\    
}
with $\WH{\MC{M}}^{-1/2} \in \R^{m \times m}$. We employ the decomposition in \eqref{klspde:decomp:alternative} to couple the KL and SPDE samplers across multiple levels. Figure \ref{fig:klspde:multilevelgrfs} shows a KL realization of the GRF in Fig. \ref{fig:spde:multilevelgrfs} sampled using the coupled KL-SPDE approach. It can be seen that additional, detailed structure is captured on the coarse level by increasing number of KL modes. For a KL-SPDE sampler with $\mathpzc{m_\ell}$ modes, the dimensionality of the complement sample space on the subsequent finer level is $\mathpzc{n}_L - \mathpzc{m}_\ell$. Hence, the high-frequency components in the complement space, $\Theta_L \backslash \Thetah$, are more visible for configuration with fewer modes on the coarse level.

\begin{figure}[H] \centering
\begin{subfigure}[h]{0.3\textwidth} %
\includegraphicsifexists[trim={4cm 4cm 4cm 4cm},clip,width=\textwidth]{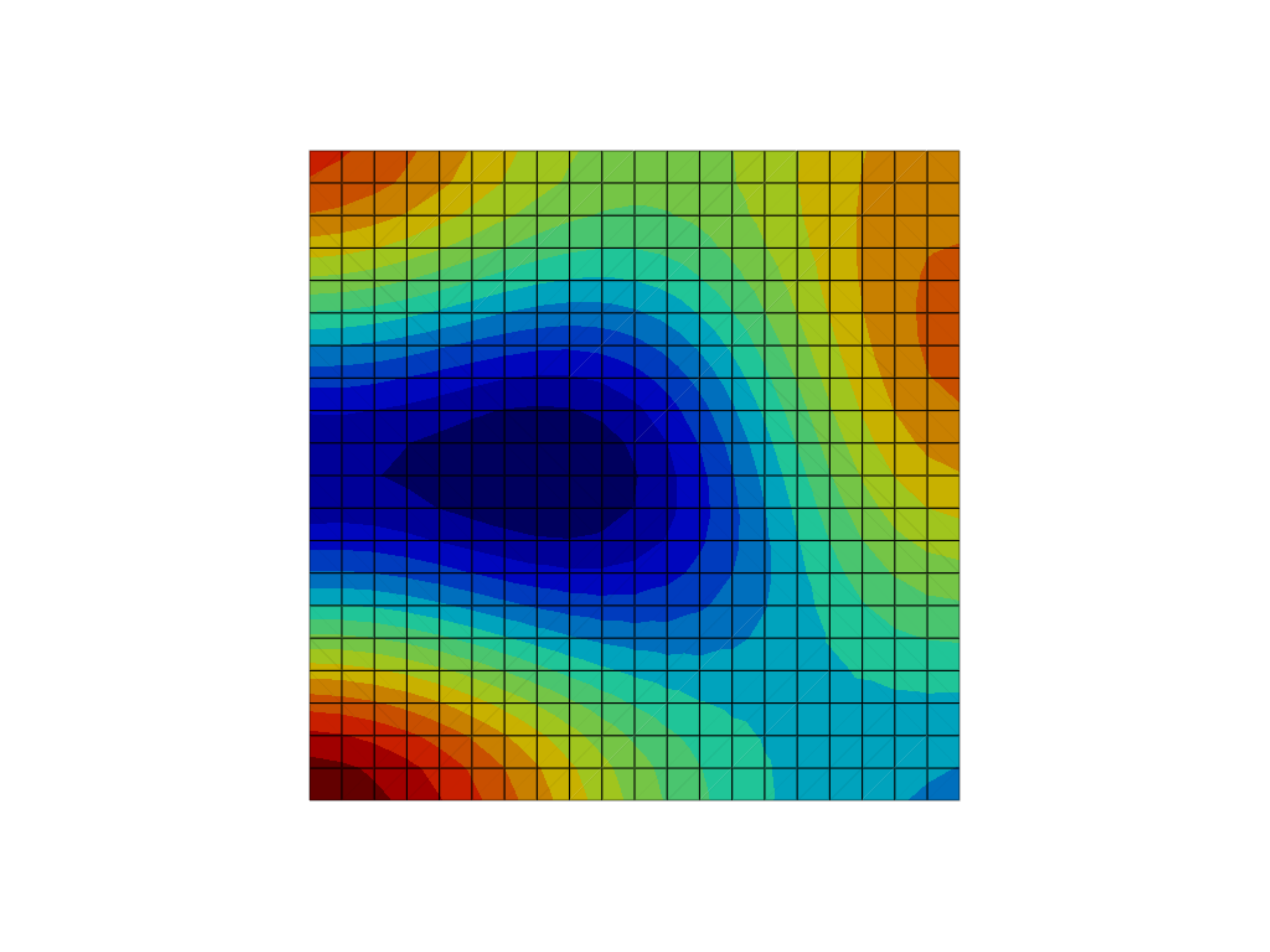} \caption{$\WH{\theta}_0 \in \WH{\Theta}_{|\Psih|}$ with $|\Psih|=10$}\label{fig:kl:10modes:l0}
\end{subfigure}
\begin{subfigure}[h]{0.3\textwidth} %
\includegraphicsifexists[trim={4cm 4cm 4cm 4cm},clip,width=\textwidth]{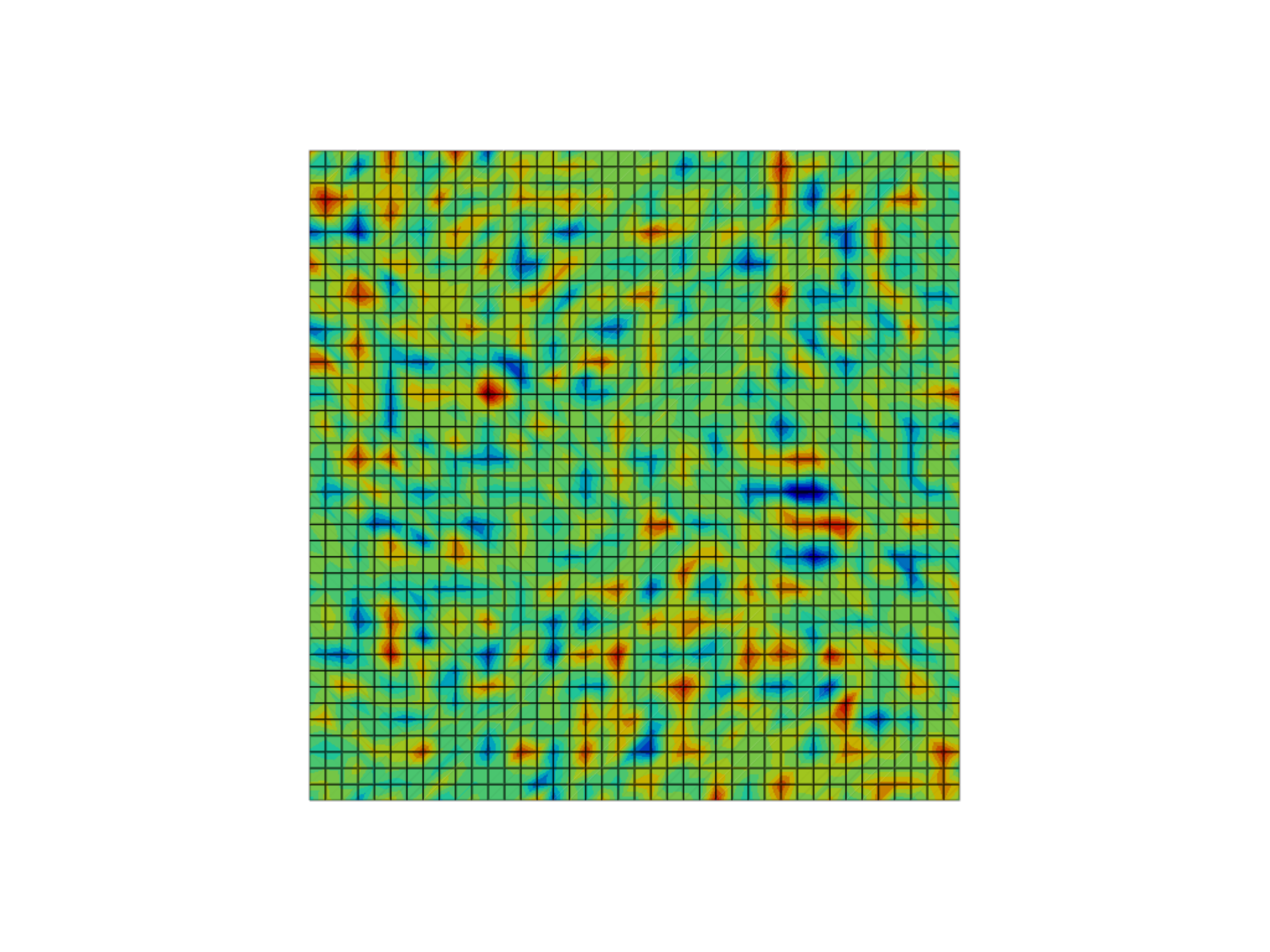} \caption{$\theta^\perp_1 \in \Theta_1 \backslash \WH{\Theta}_{10}$}\label{fig:kl:10modes:complement:l0}
\end{subfigure}
\begin{subfigure}[h]{0.3\textwidth} %
\includegraphicsifexists[trim={4cm 4cm 4cm 4cm},clip,width=\textwidth]{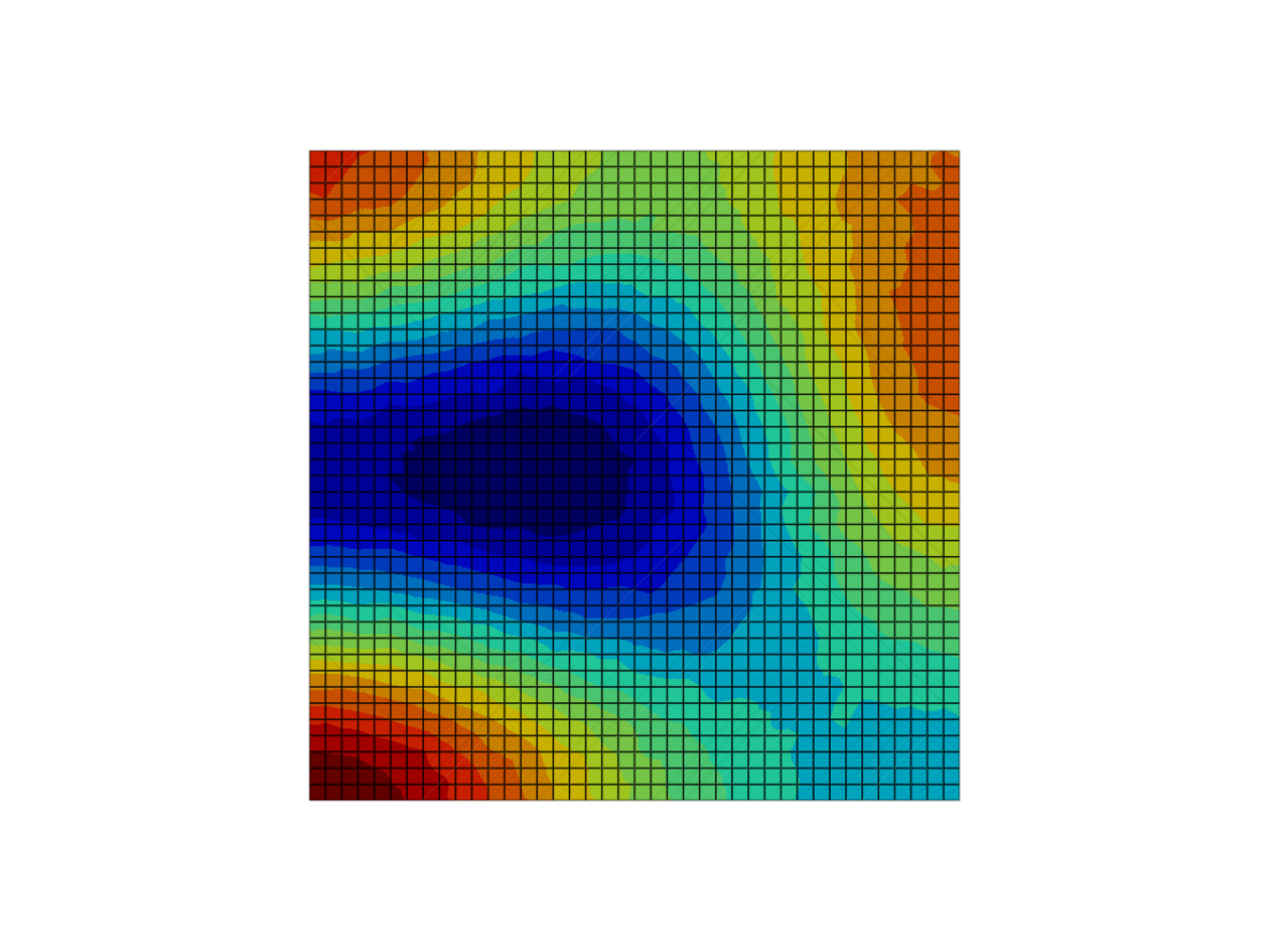} \caption{$\widetilde{\theta}_1 = \WH{\theta}^1_0 + \theta^\perp_1 \in \Theta_1$}\label{fig:kl:10modes:l1}
\end{subfigure}
\begin{subfigure}[h]{0.3\textwidth} %
\includegraphicsifexists[trim={4cm 4cm 4cm 4cm},clip,width=\textwidth]{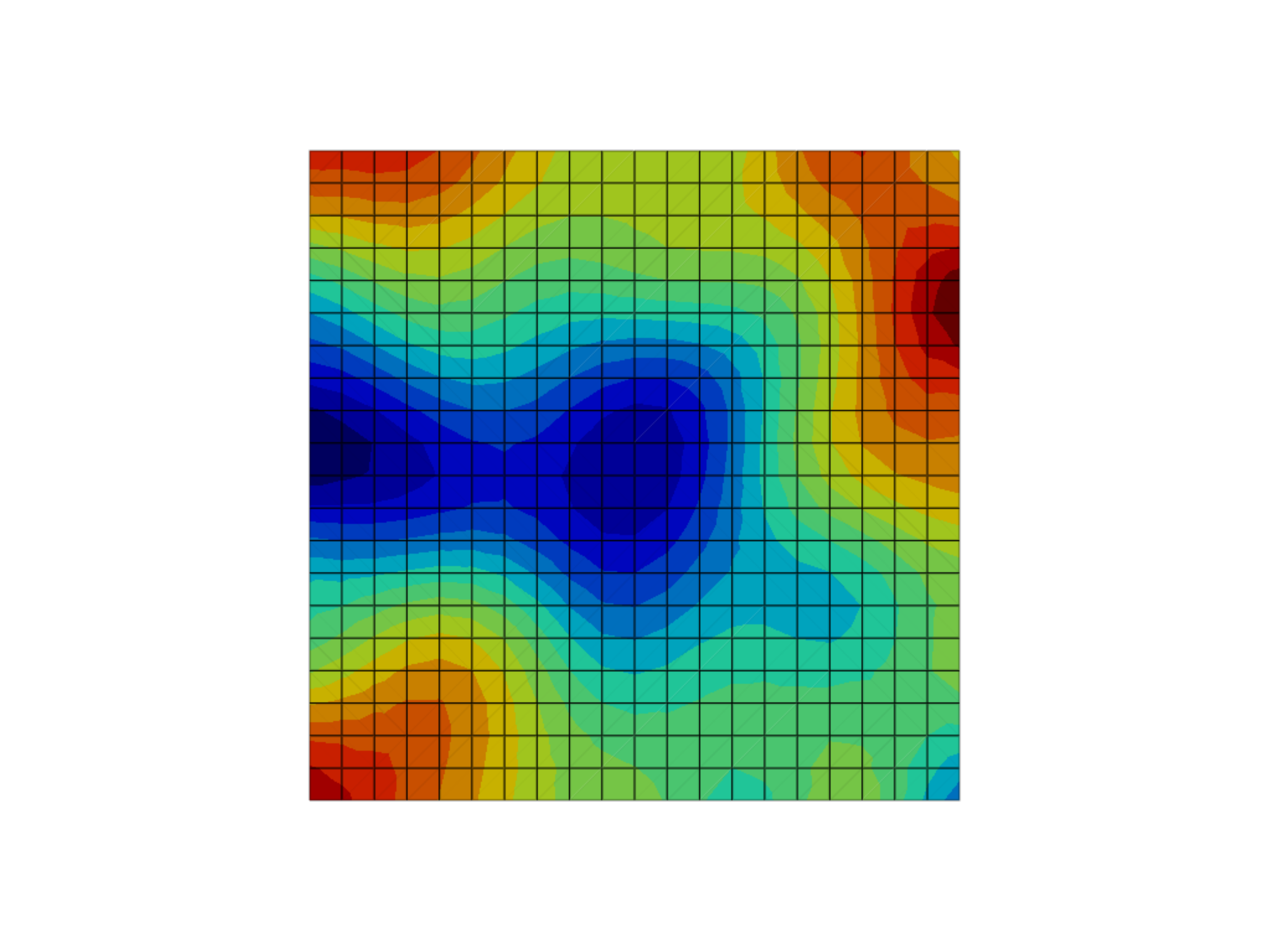} \caption{$\theta_0 \in \WH{\Theta}_{|\Psih|}$ with $|\Psih|=50$}\label{fig:kl:50modes:l0}
\end{subfigure}
\begin{subfigure}[h]{0.3\textwidth} %
\includegraphicsifexists[trim={4cm 4cm 4cm 4cm},clip,width=\textwidth]{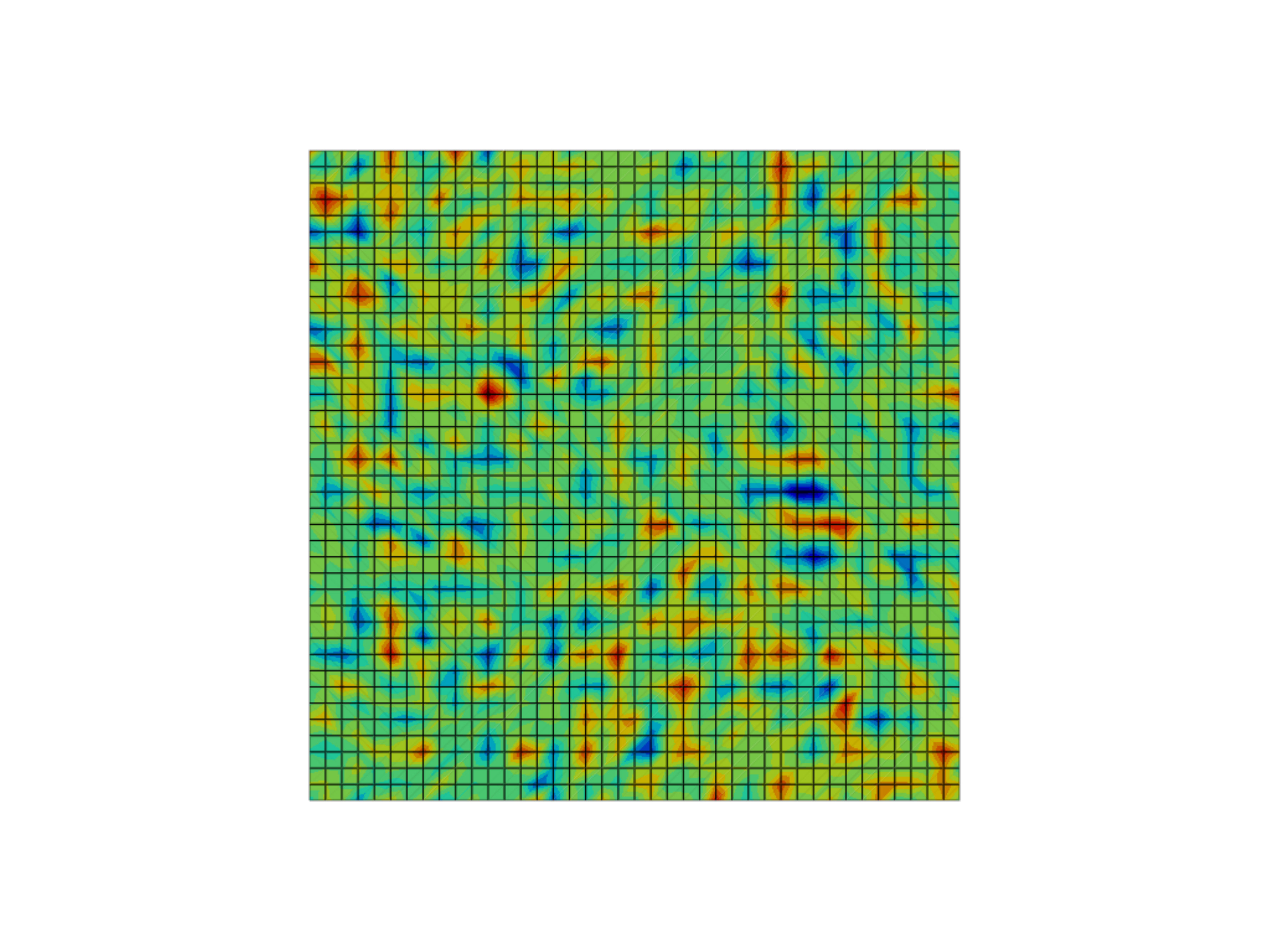} \caption{$\theta^\perp_1 \in \Theta_1 \backslash \WH{\Theta}_{50}$}\label{fig:kl:50modes:complement:l0}
\end{subfigure}
\begin{subfigure}[h]{0.3\textwidth} %
\includegraphicsifexists[trim={4cm 4cm 4cm 4cm},clip,width=\textwidth]{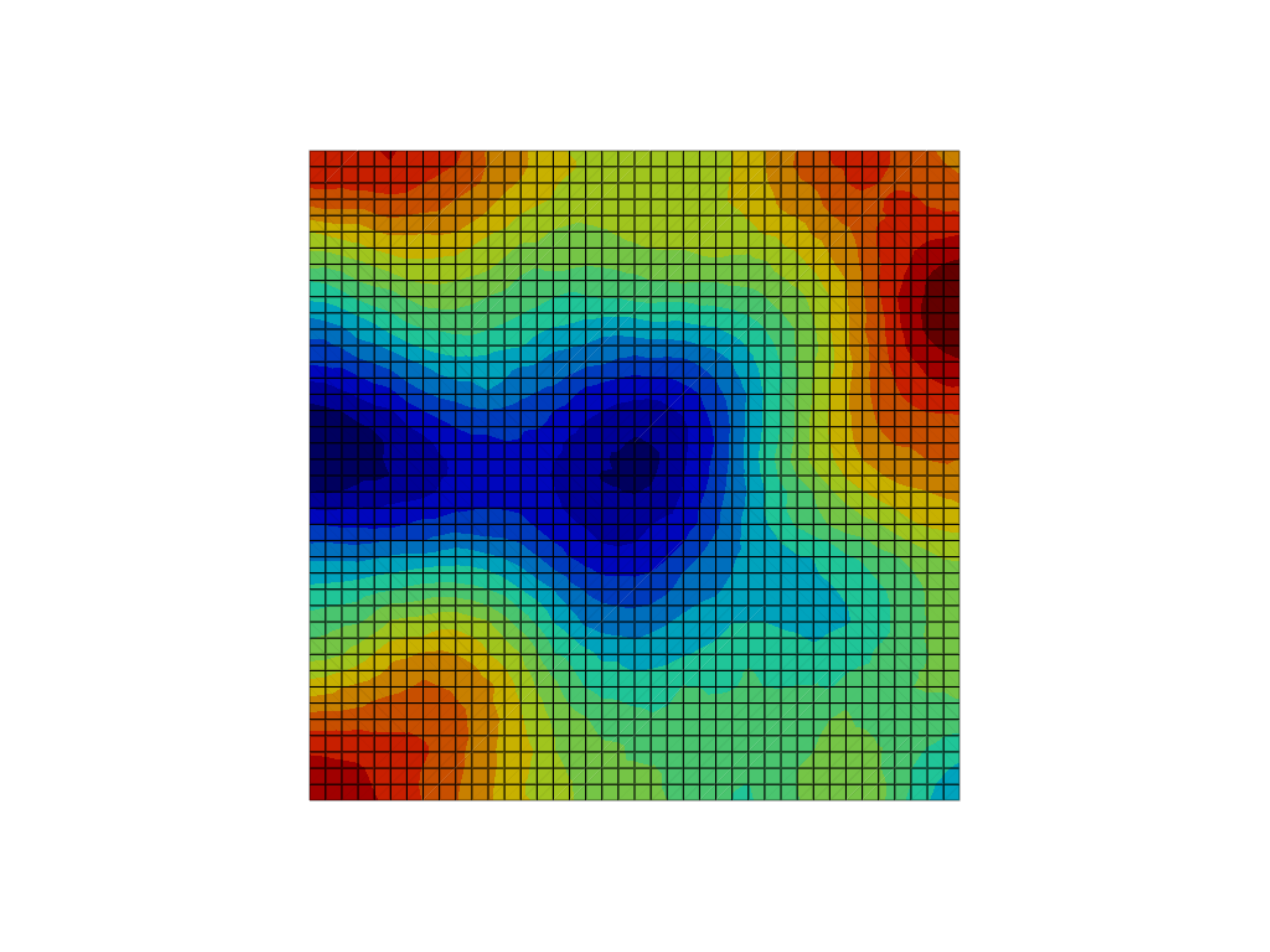} \caption{$\widetilde{\theta}_1 = \WH{\theta}^1_0 + \theta^\perp_1 \in \Theta_1$}\label{fig:kl:50modes:l1}
\end{subfigure}
\begin{subfigure}[h]{0.3\textwidth} %
\includegraphicsifexists[trim={4cm 4cm 4cm 4cm},clip,width=\textwidth]{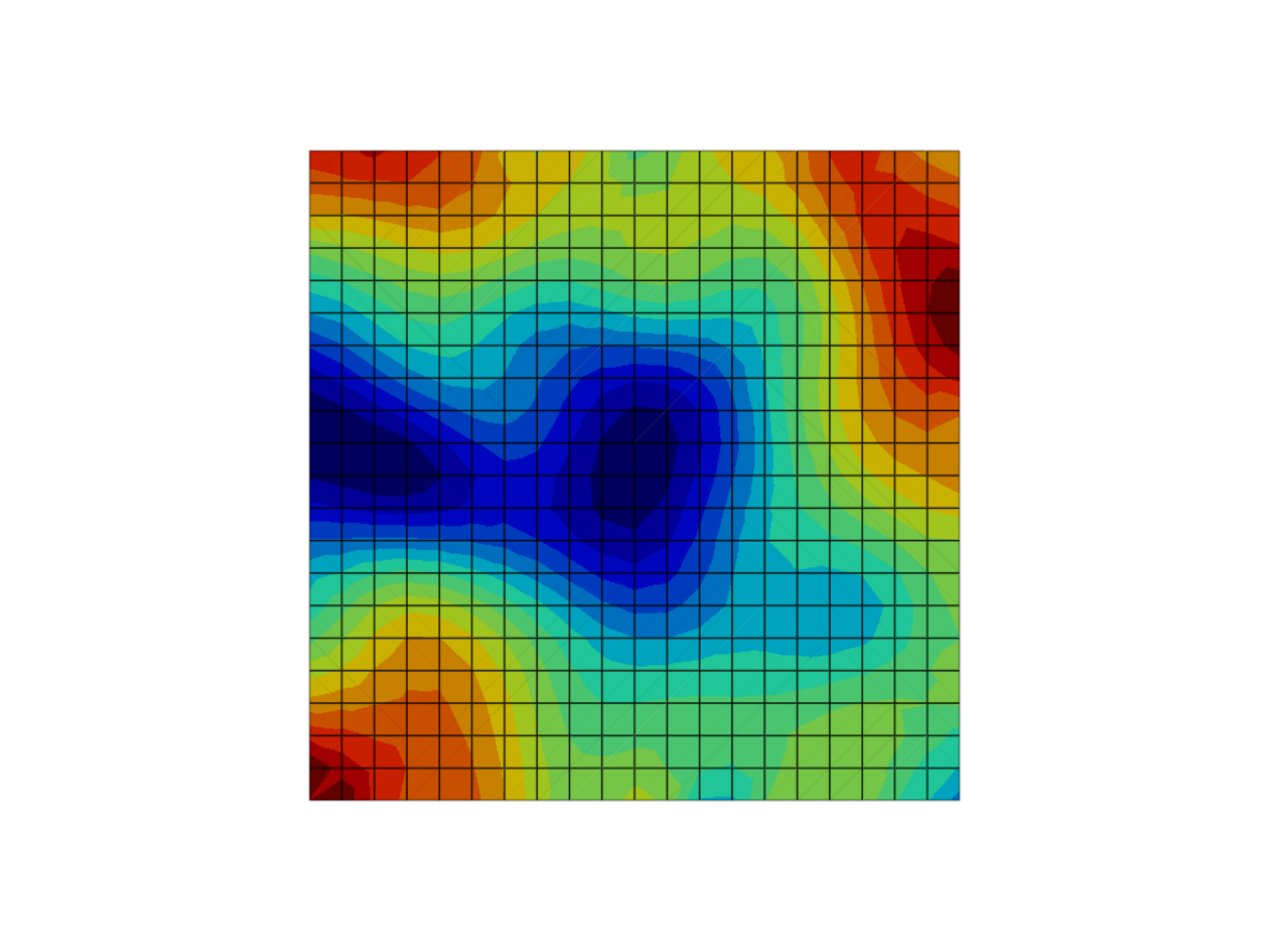} \caption{$\theta_0\in \WH{\Theta}_{|\Psih|}$ with $|\Psih|=100$}\label{fig:kl:100modes:l0}
\end{subfigure}
\begin{subfigure}[h]{0.3\textwidth} %
\includegraphicsifexists[trim={4cm 4cm 4cm 4cm},clip,width=\textwidth]{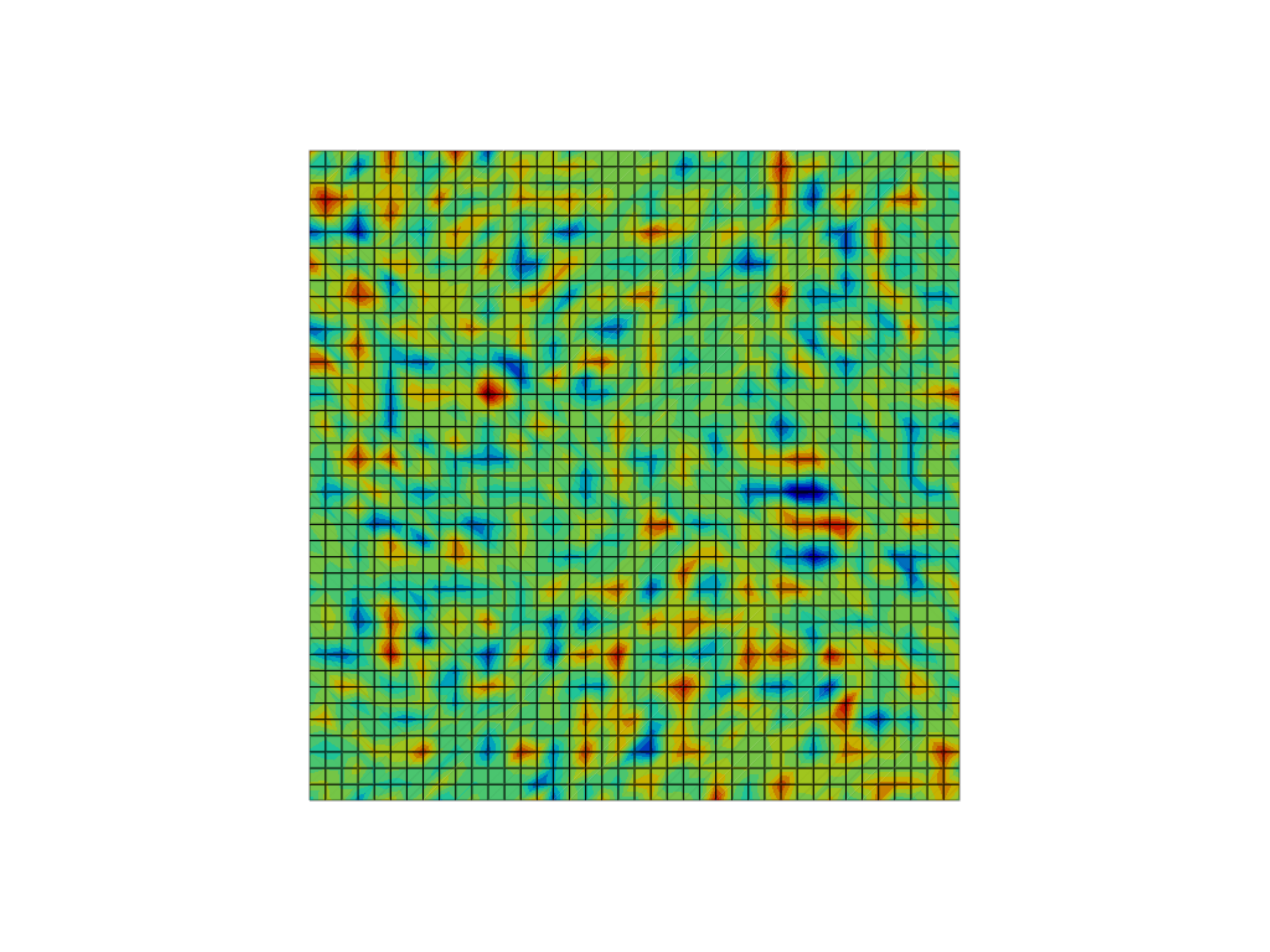} \caption{$\theta^\perp_1 \in \Theta_1 \backslash \WH{\Theta}_{100}$}\label{fig:kl:100modes:complement:l0}
\end{subfigure}
\begin{subfigure}[h]{0.3\textwidth} %
\includegraphicsifexists[trim={4cm 4cm 4cm 4cm},clip,width=\textwidth]{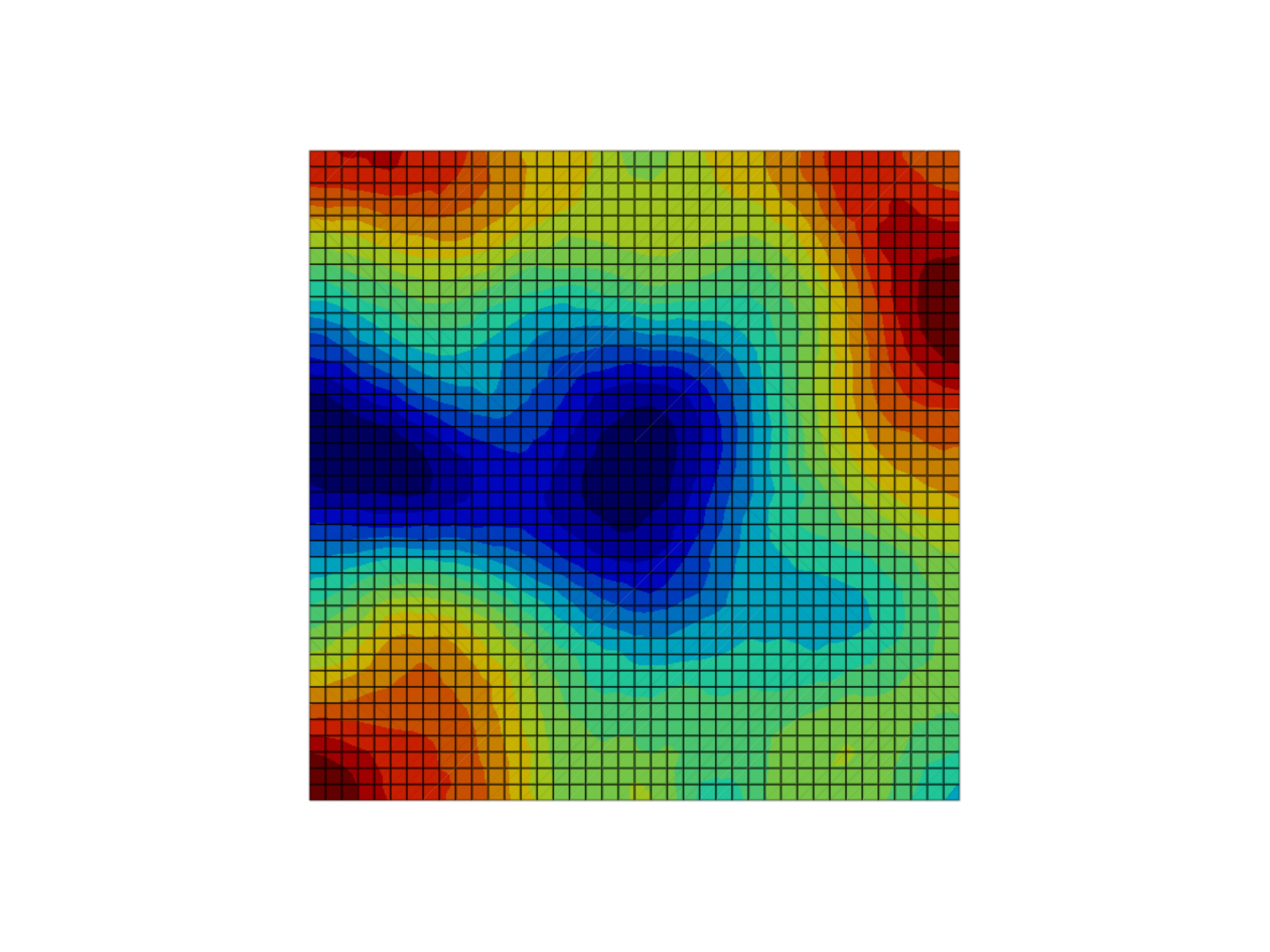} \caption{$\widetilde{\theta}_1 = \WH{\theta}^1_0 + \theta^\perp_1 \in \Theta_1$}\label{fig:kl:100modes:l1}
\end{subfigure}
\captionsetup{singlelinecheck=off,font=footnotesize}
\caption[]{A realization of a Gaussian random field across the levels of a two-level hierarchy sampled using the \KL and SPDE sampler on the coarse and fine level, respectively. on the  approach with: (a,d,g): a sample on the coarse level, (b,e,h): a fine-level sample in the space complement to $\WH{\Theta}_{\mathpzc{m}} = \mathrm{span}(\Psih)$, and (c,f,i): the complete multilevel realization of the GRF. $\WH{\Theta}_{\mathpzc{m}}$ represents the use of $\mathpzc{m}$ modes on the coarse level (i.e. $\mathpzc{m} = |\Psih|$).}
\label{fig:klspde:multilevelgrfs}
\end{figure}

\section{Bayesian Inference: Multi-Level Markov Chain Monte Carlo} \label{sec:MCMC}

The Bayesian framework provides a systematic approach to infer unknown model parameters by combining prior knowledge with observed data, $\yobs$. Starting from Bayes' theorem, several techniques, both deterministic (e.g., Laplace approximation, maximum likelihood estimation (MLE), maximum-a-posteriori (MAP) estimation \cite{Bassett2018}, and variational inference \cite{Blei2018}) and stochastic/sampling (e.g., Monte Carlo) based, have been developed to estimate or sample from the posterior distribution. To demonstrate our hierarchical sampling approach, we focus on Markov Chain Monte Carlo (MCMC) methods, which are widely used for Bayesian inference due to their ability to sample from complex, high-dimensional distributions. MCMC methods generate a Markov chain whose stationary distribution is the target posterior distribution. The chain is constructed by proposing new states based on the current state and accepting or rejecting these proposals based on a Metropolis-Hastings acceptance criterion. The chain converges to the target distribution asymptotically (i.e. as the number of samples approaches infinity). This work considers multilevel MCMC sampling which samples on the the coarser levels to inform sampling on the finer levels (often referred to as \emph{Delayed Acceptance}), while simultaneously improving and accelerating the convergence of multilevel estimation of statistical moments.

For multilevel MCMC sampling from a posterior distribution, we consider a hierarchy of maps $\cbrac{\mathcal{F}_\ell}_{\ell=0}^{L}$ with increasing fidelity. We employ a hierarchy of increasingly refined discretizations (i.e., mesh) of a single PDE, that naturally stem from multigrid methods. Hence, $\ell$ denotes the order of refinement, with levels $\ell=0,1,\hdots,L$, such that level $L$ is the finest level. Each $\mathcal{F}_\ell := \mathcal{D}_\ell \circ \mathcal{S}_\ell(\xi_\ell)$ maps the parameters being inferred $\xi$ to model output $\Vec{y}$ and is a composition of $\mathcal{S}_\ell: \xi_\ell \mapsto \theta_\ell$ corresponding to the sampler (i.e. KL and SPDE samplers) for generating GRFs from standard white noise, and $\mathcal{D}_\ell: \theta_\ell \mapsto \Vec{y}_\ell$ corresponding to the forward model (e.g. Darcy solver). For a given level $\ell$, the associated posterior is denoted as $\pi_\ell(\xi) \equiv \pi(\xi_\ell | \Vec{y}_{obs}) \propto \Lhood_\ell(\Vec{y}_{obs}|\xi_\ell) \nu_\ell(\xi_\ell)$, with prior $\nu_\ell(\xi_\ell)$ and likelihood $\Lhood_\ell(\yobs|\xi_\ell)$.

Our goal is to estimate the posterior mean of a scalar quantity of interest (QoI) $Q_{L}:=Q_L(\xi_L)$
\expression{
    \mathbb{E}_{\pi_L}[Q_L] = \integral{\Omega}{}{Q_L(\xi)\pi_L(\xi)~d\xi}
}
by means of Monte Carlo (MC) sampling by drawing samples, $\xi_L^{(i)}$, from the posterior using MCMC
\expression{
    \mathbb{E}_{\pi_L}[Q] \approx \overline{Q}_L = \dfrac{1}{N_L} \sum_{i=1}^{N_L} Q_L(\xi_L^{(i)}),
}
where, $\xi_L^{(i)} \sim \pi_L(\xi_L | \Vec{y}_{obs})$ is the $i^{th}$ independent sample and $N_L$ is the total number of samples on level $L$. The multilevel decomposition of this expectation is given as
\eq{ML:Expectation}{
    \mathbb{E}_{\pi_L}[Q_L] = \mathbb{E}_{\pi_0}[Q_0] + \sum_{\ell=1}^{L} \rbrac{\mathbb{E}_{\pi_{\ell}}[Q_{\ell}] - \mathbb{E}_{\pi_{\ell-1}}[Q_{\ell-1}]}
}
where, we use the estimator $\widehat{Y}_\ell^{N_\ell}$ for the difference $\mathbb{E}_{\pi_\ell}[Q_\ell] - \mathbb{E}_{\pi_{\ell-1}}[Q_{\ell-1}]$, defined as
\eq{ML:Y}{
    \widehat{Y}_\ell^{N_\ell} = \dfrac{1}{N_\ell} \sum_{i=1}^{N_\ell} Y^{(i)}_\ell = \dfrac{1}{N_\ell} \sum_{i=1}^{N_\ell} \rbrac{Q^{(i)}_{\ell} - Q^{(i)}_{\ell-1}}.
}
Multilevel Monte Carlo and Multilevel MCMC obtain variance reduction by sampling each $Y_\ell^{(i)}= Q^{(i)}_{\ell} - Q^{(i)}_{\ell-1}$ from a joint distribution, with expectation and variance denoted as $\mathbb{E}_{\pi_\ell,\pi_{\ell-1}}[Y_\ell]$ and $\Var[\pi_\ell,\pi_{\ell-1}]{Y_\ell}$, respectively. Previous work by Dodwell \etal~\cite{Dodwell2015,Dodwell2019} shows empirical evidence of this variance decay; specifically, they prove the multilevel acceptance probability (which is later referred to as $\alpha$) tends to one with mesh refinement, and consequently, the variance of $Y_\ell$ tends to zero (as $\ell \rightarrow \infty$). As described in Section~\ref{sec:mldecomposition}, we obtain these samples via the hierarchical PDE sampler approach. 

The decay in estimate of $\Var[\pi_\ell,\pi_{\ell-1}]{Y_\ell}$ with level refinement directly affects the sampling cost, with faster decay requiring fewer simulations on the finer levels and faster convergence. For a given MSE tolerance $\epsilon^2$, the optimal number of samples on each level, $N_\ell$, is calculated as
\eq{EffectiveSampleSize}{
    N_\ell = \dfrac{2}{\epsilon^2} \rbrac{\sum_{k=0}^L \sqrt{\Var[\pi_k,\pi_{k-1}]{Y_k} C^{\eff}_k}} \sqrt{\dfrac{\Var[\pi_\ell,\pi_{\ell-1}]{Y_\ell}}{C^{\eff}_\ell}},
}
where, $C^{\eff}_\ell$ for $\ell=0,1,\hdots L$ is the computational cost of generating an independent sample of $Y_\ell$ and is given as
\eq{EffecCost}{
    C^{\eff}_\ell := \ceil*{\tau_\ell} \rbrac{C_\ell + \ceil*{\tau_{\ell-1}} C_{\ell-1}}
}
with an integrated autocorrelation time (IACT), $\tau_\ell$, on level $\ell$. The IACT is the number of samples needed to generate an independent sample and is defined as
\eq{IACT}{
    \tau = 1 + 2 \sum_{\chi=1}^M \rho_Q(\chi),
}
where, the normalized autocorrelation function is estimated as
\eq{NormAutoCorr}{
    \rho_Q(\chi) = \dfrac{1}{N-\chi} \sum_{i=1}^M \dfrac{\rbrac{Q^{(i)} - \mu_Q} \rbrac{Q^{(i+\chi)} - \mu_Q}}{\sigma^2_Q}
}
with mean $\mu_Q$ and variance $\sigma_Q$ of $\cbrac{Q^{(i)}}_{i=1}^N$ and $M \ll N$. For details on this variance-cost relation, we refer the reader to \cite[Appendix A.2]{Reddy2025:MLMCMC}. 

The posterior is sampled using a Multi-Level Delayed Acceptance (MLDA) algorithm (Alg. \ref{Alg:HMCMC}) and we generate proposals $\xi_\ell^P$ on level $\ell$, conditioned on the current state of the chain $\xi_\ell^C$, using preconditioned Crank-Nicolson (pCN)
\expression{
    \xi_\ell^P := \sqrt{1-\beta^2} \xi_\ell^C + \beta \breve{\xi}_\ell,
}
where $\breve{\xi}_\ell \sim \MC{N}(0,I)$ is the standard white noise and $\xi_\ell^P \sim q(\xi|\xi_\ell^C) = \mathcal{N}(\sqrt{1-\beta^2} \xi_\ell^C, \beta^2 I)$. Then, the multilevel realization of $\MC{W}$ on level $\ell$, $\widetilde{\MC{W}}_\ell$, conditioned on the coarse level $\xi_{\ell-1}$, is computed using \eqref{hpde:decomp} and \eqref{klspde:decomp:alternative}.

\begin{remark}\normalfont
    The preconditioned Crank-Nicolson proposal $q(\xi^P|\xi^C) \sim \mathcal{N}(\sqrt{1-\beta^2} \xi_\ell^C, \beta^2 \mathcal{C})$ with a prior $\nu \sim \mathcal{N}(0,\mathcal{C})$ satisfies $\nu(\xi^C)q(\xi^P|\xi^C) = \nu(\xi^P)q(\xi^C|\xi^P)$.
\end{remark}

\begin{algorithm}[h]
    \caption{MLDA Algorithm: Generate a single sample on the fine level $L$ conditioned on the coarse level $\ell = L-1$.} \label{Alg:HMCMC}
    \algorithmicrequire{~Current state: $\xi_\ell^C$, $\widetilde{\xi}^C_\ell$, $\xi_L^C$, $\widetilde{\xi}^C_L$} \\
    \algorithmicensure{~Next state: $\xi_L^\star$, $\xi_L^C$}
    \begin{algorithmic}
    \State Generate $\xi_\ell^\star$ on the coarse level using MLDA using $\xi_\ell^C$ and $\widetilde{\xi}^C_\ell$ \Comment{Generate sample on the coarse level}
    \State $\xi_L^P \sim q(\xi_L|\xi_L^C)$ \Comment{Generate fine-level proposal using pCN} 
    \State Compute $\widetilde{\xi}_L$ using \eqref{hpde:decomp} or \eqref{klspde:decomp:alternative} with $\xi_L^P$ and $\xi_\ell^\star$ \Comment{Generate multilevel realization of noise}
    \State $\alpha = \min\cbrac{1,\dfrac{\Lhood_L(\yobs|\widetilde{\xi}^P_L) \Lhood_\ell(\yobs|\xi_\ell^C)}{\Lhood_L(\yobs|\xi_L^C)\Lhood_\ell(\yobs|\xi_\ell^\star)}}$ \Comment{Compute multilevel acceptance ratio} \\
    \If{$\alpha \geq u \sim \mathcal{U}(0,1)$} \Comment{Accept proposed sample}
        \State $\xi_L^\star = \widetilde{\xi}_L$  \Comment{Update the multilevel state up to current level}
        \State $\xi_L^C = \xi_L^P$ \Comment{Update the state on current level}        
    \Else \Comment{Reject the proposed sample}
        \State $\xi_L^\star = \widetilde{\xi}^C_L$
    \EndIf        
    \end{algorithmic}    
\end{algorithm}

\begin{remark}\normalfont
    On the coarsest level $\ell=0$, Alg. \ref{Alg:HMCMC} reduces to a standard, single-level Metropolis-Hastings algorithm with $\widetilde{\xi}_0 \sim \MC{N}(0,I)$ and
    \[\alpha = \min\cbrac{1,\dfrac{\Lhood_0(\yobs|\xi_0^P)}{\Lhood_0(\yobs|\xi_0^C)}}.\]
    Furthermore, when the coarsest level is sampled using the SPDE approach, then $\xi_0 \in \R^{\mathpzc{n_0}}$, where $\mathpzc{n_\ell}$ is the number of degrees of freedom (DOFs) on level $\ell$. When the coarsest level is sampled using the \KL approach, then $\xi_0 \in \R^{\mathpzc{m_0}}$, where $\mathpzc{m_\ell}$ is the number of \KL modes on level $\ell$.
\end{remark}

\section{Subsurface Flow: Darcy's Equations} \label{sec:Darcy}

We demonstrate our MLMCMC approach on a typical problem in groundwater flow described by Darcy's Law, which has been employed extensively in verifying the efficiency of delayed acceptance MCMC algorithms \cite{Dodwell2015,Dodwell2019,Fairbanks2021,Teckentrup13,Lykkegaard2023,lykkegaard2021accelerating,Cliffe2011}. The governing equations with the appropriate boundary conditions are written as
\subeqs{GovEq}{
\begin{align}
    \spliteq{Darcy}{    
        \dfrac{1}{k(\Vec{x})}\Vec{u}(\Vec{x}) + \Grad{p(\Vec{x})} = \Vec{f}& \quad \mathrm{in~} \Omega 
    }\\
    \spliteq{Mass}{
        \Div{\Vec{u}(\Vec{x})} = 0& \quad \mathrm{in~} \Omega
    }\\
    \spliteq{BC:P}{
        p = p_s & \quad \mathrm{on~} \Gamma_D
    }\\
    \spliteq{BC:NoFlux}{
    \Vec{u}\cdot\Vec{n} = 0 & \quad \mathrm{on~} \Gamma_N
    }
\end{align}
}
with $\Gamma := \Gamma_D \bigcup \Gamma_N = \partial \Omega$ where \eqref{Mass} is the incompressibility condition. The permeability coefficient is assumed to follow log-normal distribution $k(\Vec{x}) = \exp(\theta(\Vec{x}))$ where $\theta(\Vec{x})$ is a realization of a mean-zero GRF with a covariance kernel defined by \eqref{matern}.
We solve the Darcy equation using the mixed finite element method and consider the function spaces
\expression{
    \Vec{R} = H(\mathrm{div},\Omega) := \cbrac{\Vec{u} \in \Vec{L}^2(\Omega)~|~\mathrm{div}~\Vec{u} \in L^2(\Omega),\Vec{u}\cdot \Vec{n} = 0~\mathrm{on}~\Gamma}
}
\expression{
    \Phi = L^2(\Omega)
}
where $\Vec{L}^2(\Omega) = [L^2(\Omega)]^d$ is the $d$-dimensional vector function space.  Denote by $\Omega_h := \bigcup_{i=1}^M K_i$ a partition of $\Omega$ into a finite collection of $M$ non-overlapping elements $K_i$ with $h := \underset{j=1,\ldots,M}{\max} diam(K_j)$.
We take $\Vec{R}_h  \in H(\mathrm{div},\Omega_h) \subset \Vec{R}$ and $\Phi_h  \in  L^2(\Omega_h) \subset \Phi$ to be the lowest order Raviart-Thomas and piecewise constant basis functions, respectively. We consider $\Vec{u} \in \Vec{R}$ and $p \in \Phi$ and denote their finite element representations by $\Vec{u}_h \in \Vec{R}_h$ and $p_h \in \Phi_h$. Introducing test functions $\Vec{v}_h \in \Vec{R}_h$ and $q_h \in \Phi_h$, the weak, mixed form of \eqref{GovEq} reads
\begin{problem}
    Find $\rbrac{\Vec{u}_h,p_h} \in \Vec{R}_h \times \Phi_h$ such that
    \aligneq{GovEq:Mixed}{
            \rbrac{k^{-1} \Vec{u}_h,\Vec{v}_h} - \rbrac{\Div{\Vec{v}_h},p_h} &= \rbrac{\Vec{f}_h,\Vec{v}_h}, \quad &\forall \Vec{v}_h \in \Vec{R}_h \\
            \rbrac{\Div{\Vec{u}_h},q_h} &= 0, \quad &\forall q_h \in \Phi_h
    }
\end{problem}
The resulting saddle point system is then given by 
\eq{GovEq:Saddle}{
    \smatrix{ M_h(k) & B_h^T \\ B_h & \Vec{0} } \smatrix{ \Vec{u}_h \\ p_h } = \smatrix{ \rbrac{\Vec{f}_h,\Vec{v}_h} \\ \Vec{0} }.
}
As previously mentioned, we consider the random fields $\theta \in H^1(\Omega_h)$ but project them into $\Phi_h$ when computing the permeability coefficient, $k$. We take the quantity of interest (QoI) to be the integrated flux across $\Gamma_{out} \subset \Gamma$
\eq{MassFlux}{
    Q = \dfrac{1}{\abs{\Gamma_{out}}} \integral{\Gamma_{out}}{}{\Vec{u}\cdot \Vec{n} ~d\Gamma}.
}

\section{Results} \label{sec:results}

\subsection{Problem Formulation} \label{subsec:problem}

We demonstrate the proposed methodology on a synthetic subsurface flow problem by estimating, from noisy pressure observations, the permeability field and the posterior distribution of the mass flux through a boundary. We consider the domain $\Omega \subset \R^2$ to be a unit-square with boundary $\partial \Omega = \Gamma = \Gamma_{\mathrm{left}} \bigcup \Gamma_{\mathrm{top}} \bigcup \Gamma_{\mathrm{right}} \bigcup \Gamma_{\mathrm{bottom}}$ discretized using a structured quadrilateral mesh with different spatial resolution at each level in the hierarchy. We consider an (L+1)-level hierarchy with mesh sizes $h_\ell = 0.1\cdot 0.5^{\ell}$. We apply Dirichlet boundary conditions $p = -1~\mathrm{on}~\Gamma_{\mathrm{left}}$ and $p = 0~\mathrm{on}~\Gamma_{\mathrm{right}}$, and Neumann boundary conditions $\Vec{u}\cdot\Vec{n}=0 ~\mathrm{on}~\Gamma_{\mathrm{top}} \bigcup \Gamma_{\mathrm{bottom}}$.

The permeability field $k(\Vx)$ is a log-normal random field with zero-mean and covariance given by the Mat\'ern kernel. We take $\yobs \in \R^{100}$ to be the noise-perturbed pressure (i.e. $p(\theta_{obs}) + \epsilon$ with $\epsilon \sim \MC{N}(0,\sigma^2_\eta I)$) at uniformly sampled points in $\Omega_h$. We generate this synthetic observational data on a reference mesh with $h=0.5 h_L$ and assume that $ \mathcal{F}(\Vec{\xi}) \sim \mathcal{N}(\yobs,\sigma^2_\eta I)$ where $I$ is an identity matrix of appropriate dimension and take the noise to be $\sigma_\eta^2 = 0.01$; consequently, the likelihood function has the relation $\Lhood(\yobs|\Vec{\xi})\propto \exp (-\Vert \mathcal{F}(\Vec{\xi})-\yobs \Vert^2/2\sigma_\eta^2)$. The reference GRF and pressure field (without added noise) are shown in Figure \ref{fig:mcmc:obs}. A total of five independent chains are generated for each sampling approach and each chain is ran for 10k samples on the finest level. We consider a prior to be Gaussian random fields of Mat\'ern covariance with correlation length $\lambda = 0.3$ and a marginal variance $\sigma^2 = 0.1$. Here onwards, we refer to the hierarchy/sampling that uses the SPDE sampler on all levels as \emph{SPDE-hierarchy/sampling} and the hierarchy/sampling that uses KL sampler on the coarsest level and SPDE sampler on all finer levels as \emph{KL-SPDE-hierarchy/sampling}.

\begin{figure}[H] \centering
\begin{subfigure}[h]{0.35\textwidth} %
\includegraphicsifexists[trim={4.5cm 2cm 2.75cm 2cm},clip,width=\textwidth]{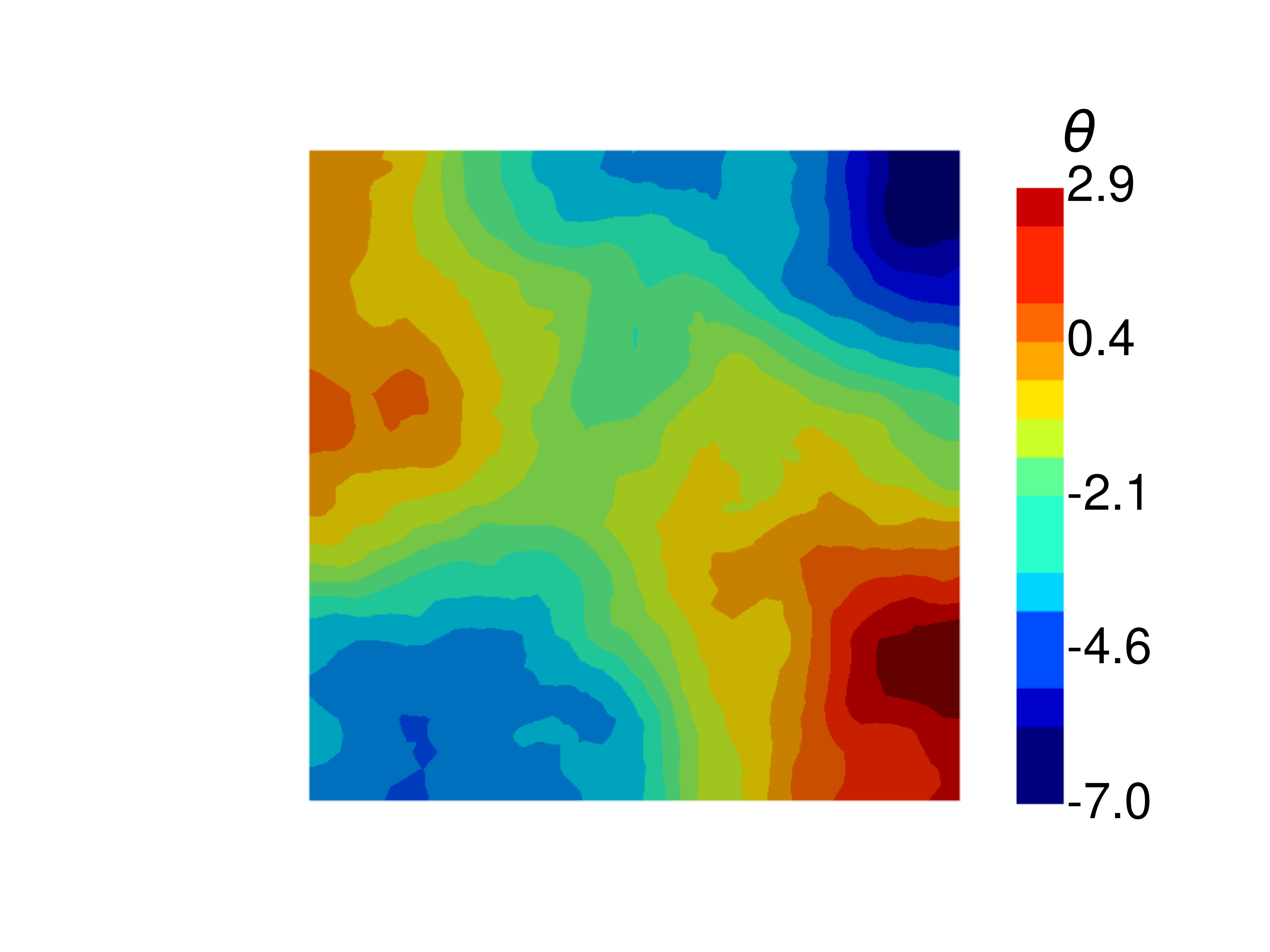} \caption{}\label{fig:mcmc:obs:grf}
\end{subfigure}
\begin{subfigure}[h]{0.35\textwidth} %
\includegraphicsifexists[trim={4.5cm 2cm 2.75cm 2cm},clip,width=\textwidth]{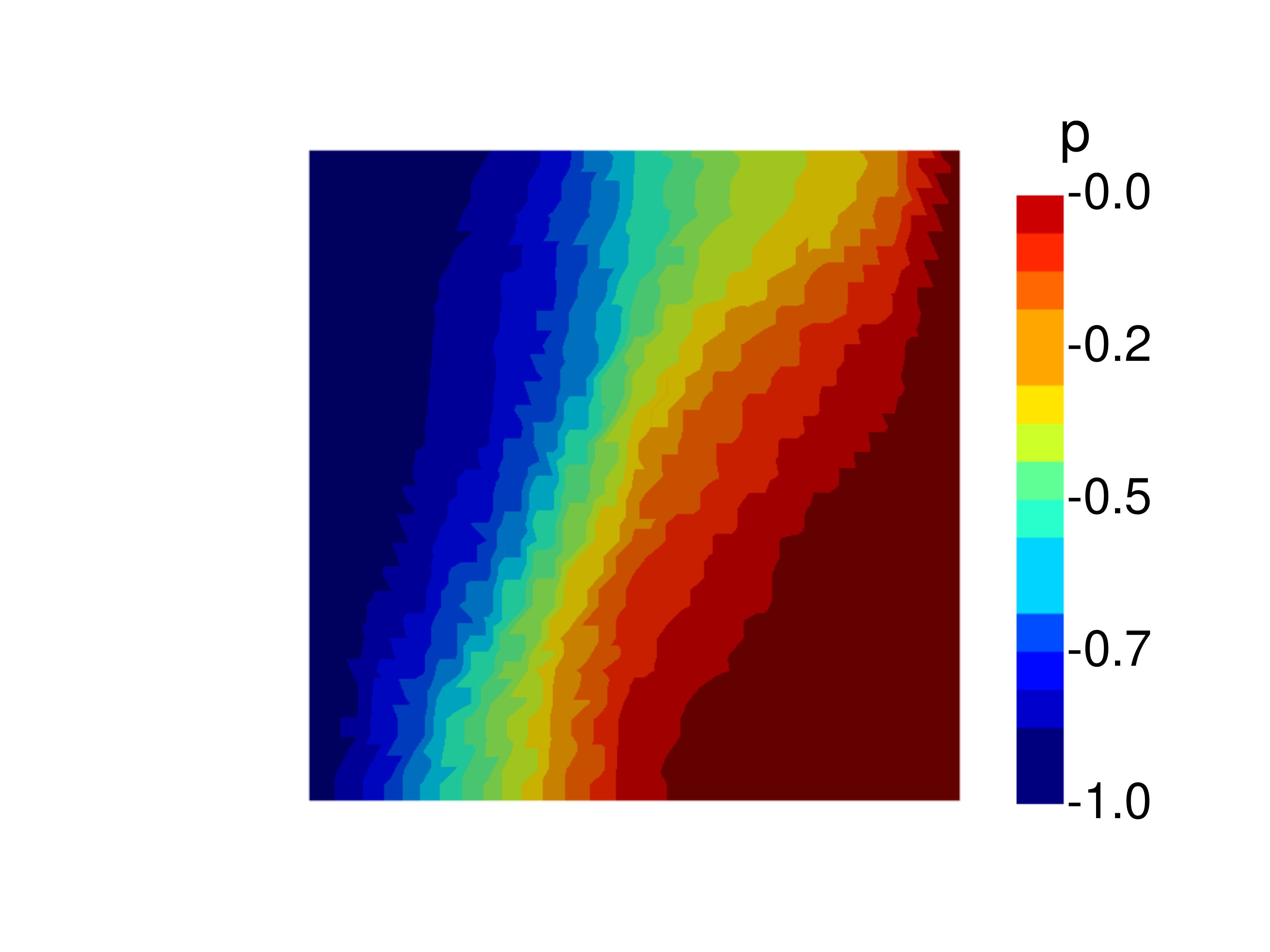} \caption{}\label{fig:mcmc:obs:p}
\end{subfigure}
\captionsetup{singlelinecheck=off,font=footnotesize}
\caption[]{The reference/observational (a) Gaussian random field and (b) the resulting pressure field.}
\label{fig:mcmc:obs}
\end{figure}

\subsection{Multilevel Bayesian Inference in Subsurface Flows} \label{subsec:bayesinf}

We compare the inferred mean GRFs and the corresponding pressure distributions obtained via MLMCMC using SPDE sampling and KL-SPDE sampling with varying number of coarse level modes and three levels in the hierarchy. Figure \ref{fig:mcmc:runs} shows that the mean GRF estimates and corresponding pressure fields on the fine level capture essential features of the \emph{observational/imposed}-GRF and the corresponding pressure field (Fig. \ref{fig:mcmc:obs}). It should be noted that the inference is performed using the perturbed/noisy pressure data, while the figures show the inferred mean pressure field (without noise). Furthermore, the insignificant differences between the fields on level $\ell=1$ and level $\ell=2$ indicate that a two-level hierarchy is sufficient to capture the essential features of the random field. On the coarse level, employing only 10 Karhunen–Loève (KL) modes is sufficient to capture the dominant features of the random fields, substantially reducing the dimensionality of the problem. This dimensionality reduction not only enhances computational efficiency but also improves the overall performance of the MCMC sampler, as fewer modes lead to a lower-dimensional sample space to explore. Consequently, the MCMC algorithm can converge more rapidly, achieving reliable estimates for the GRF and pressure fields with significantly reduced sampling effort. While the coarsest level of the KL-SPDE hierarchy with 10 KL modes (Fig. \ref{fig:mcmc:mlklpde-10:grf-l0}) captures the large-scale features of the imposed field, the subsequent finer level (Fig. \ref{fig:mcmc:mlklpde-10:grf-l1}) is able to capture additional detailed structure in the complement space. As the number of modes on the coarse level increases, the estimated GRF characteristics tend to those obtained by SPDE-sampling. By leveraging this multilevel structure with an optimized number of KL modes, the methodology can efficiently balance accuracy and computational cost, making it a practical approach for high-dimensional uncertainty quantification in complex systems.
\begin{figure}[h] \centering
\begin{subfigure}[h]{0.24\textwidth} %
\includegraphicsifexists[trim={4.5cm 2cm 2.75cm 2cm},clip,width=\textwidth]{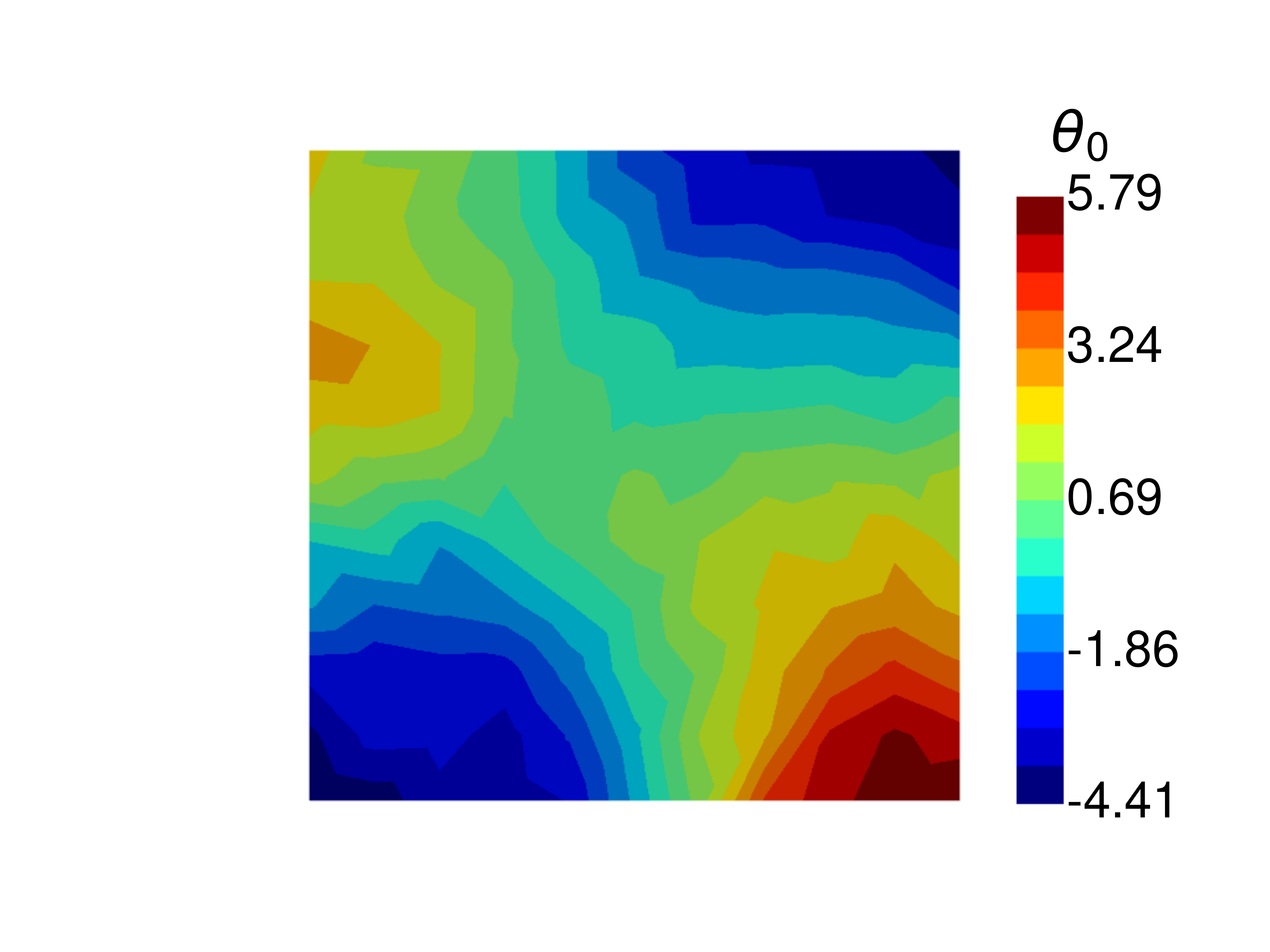} \caption{}\label{fig:mcmc:mlpde:grf-l0}
\end{subfigure}
\begin{subfigure}[h]{0.24\textwidth} %
\includegraphicsifexists[trim={4.5cm 2cm 2.75cm 2cm},clip,width=\textwidth]{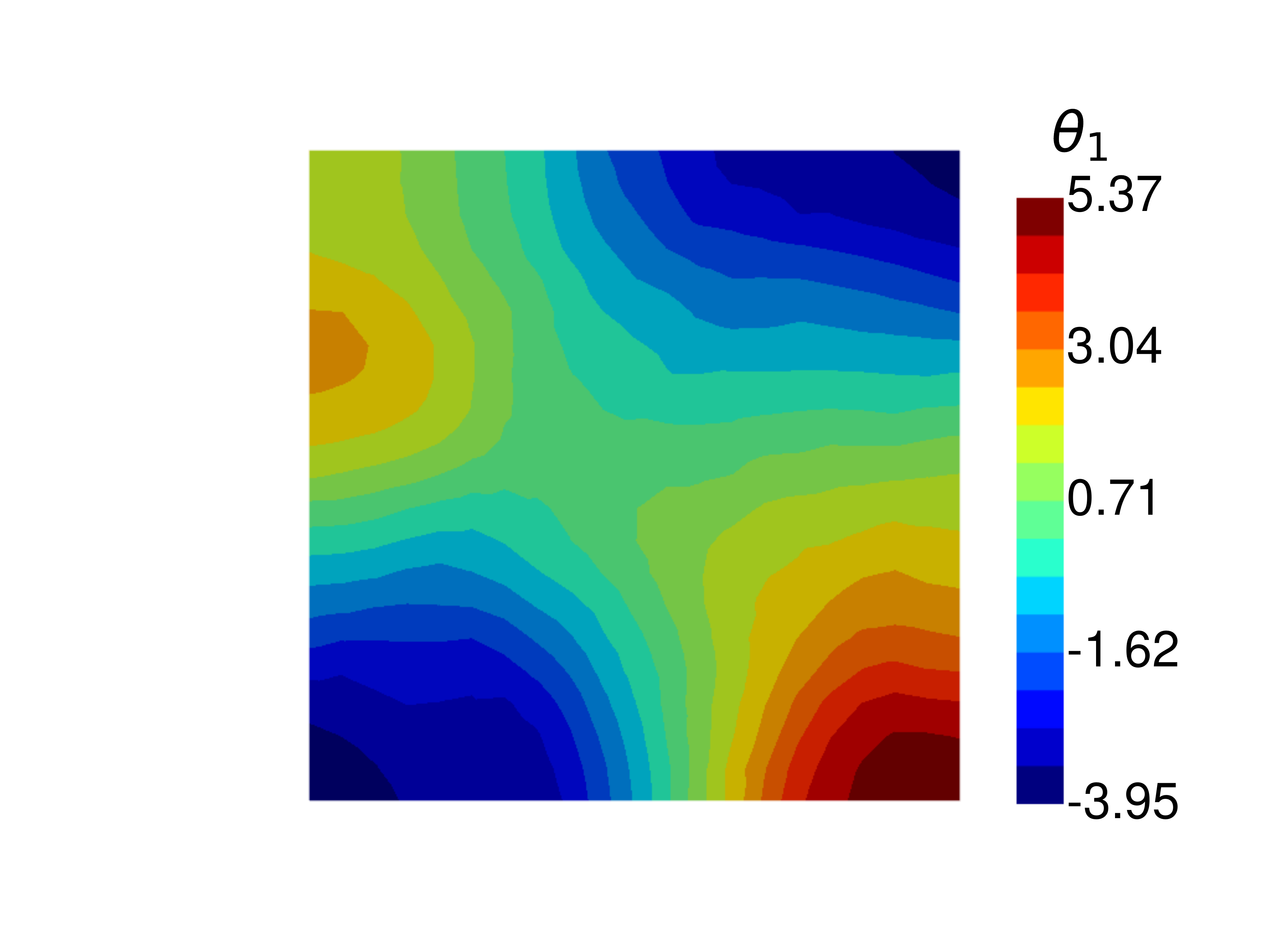} \caption{}\label{fig:mcmc:mlpde:grf-l1}
\end{subfigure}
\begin{subfigure}[h]{0.24\textwidth} %
\includegraphicsifexists[trim={4.5cm 2cm 2.75cm 2cm},clip,width=\textwidth]{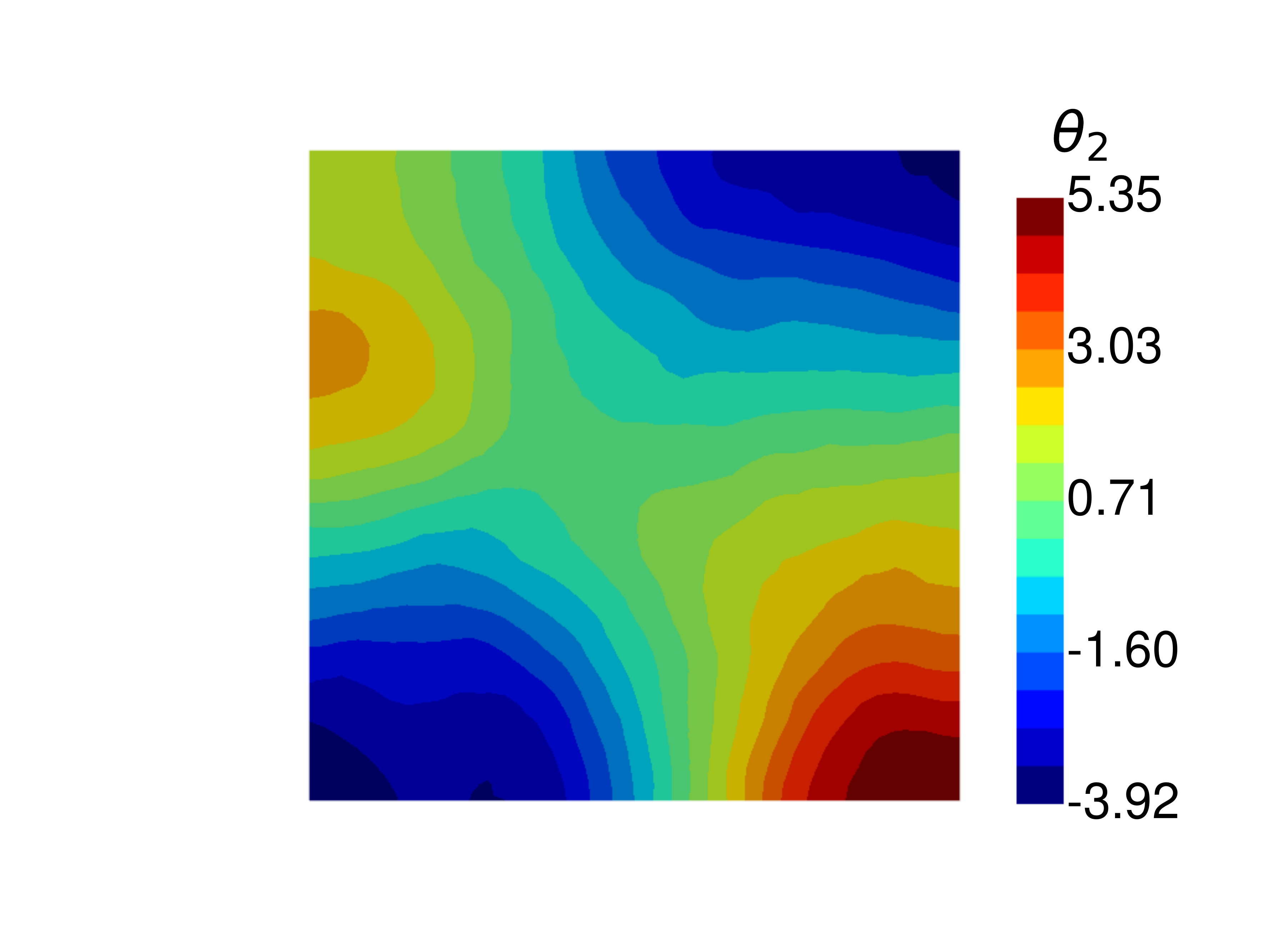} \caption{}\label{fig:mcmc:mlpde:grf-l2}
\end{subfigure}
\begin{subfigure}[h]{0.24\textwidth} %
\includegraphicsifexists[trim={4.5cm 2cm 2.75cm 2cm},clip,width=\textwidth]{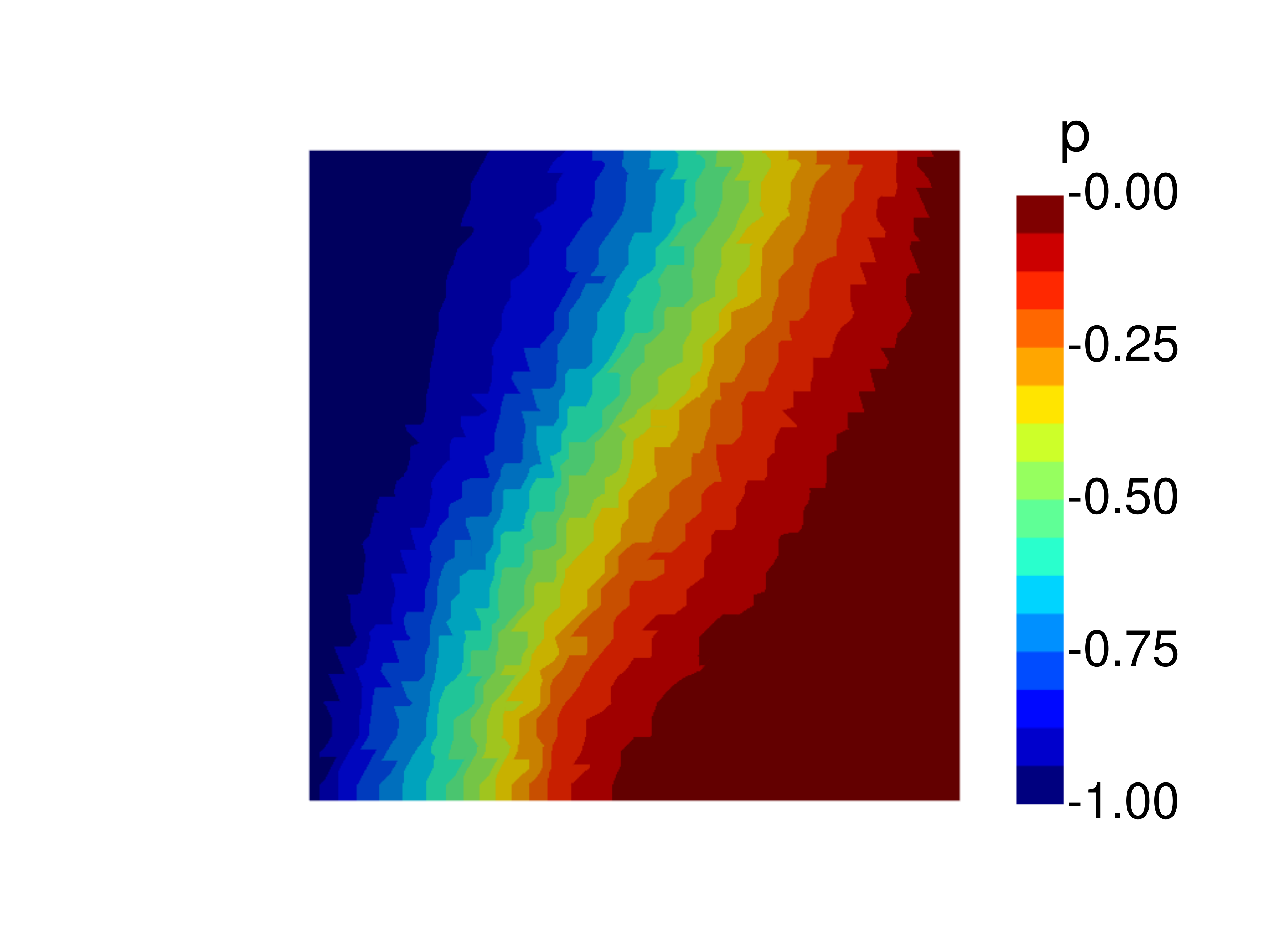} \caption{}\label{fig:mcmc:mlpde:p-l2}
\end{subfigure}
\begin{subfigure}[h]{0.24\textwidth} %
\includegraphicsifexists[trim={4.5cm 2cm 2.75cm 2cm},clip,width=\textwidth]{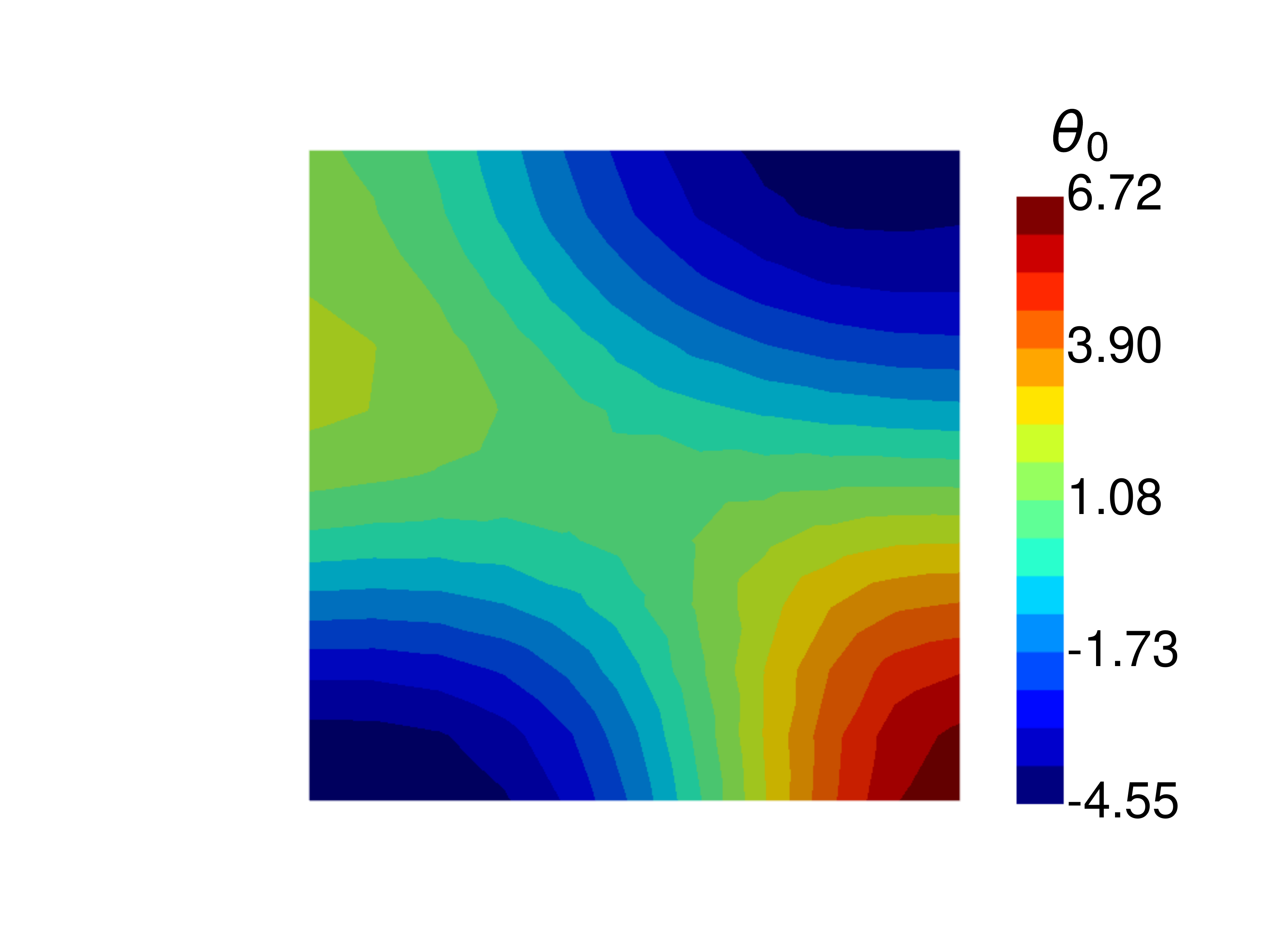} \caption{}\label{fig:mcmc:mlklpde-10:grf-l0}
\end{subfigure}
\begin{subfigure}[h]{0.24\textwidth} %
\includegraphicsifexists[trim={4.5cm 2cm 2.75cm 2cm},clip,width=\textwidth]{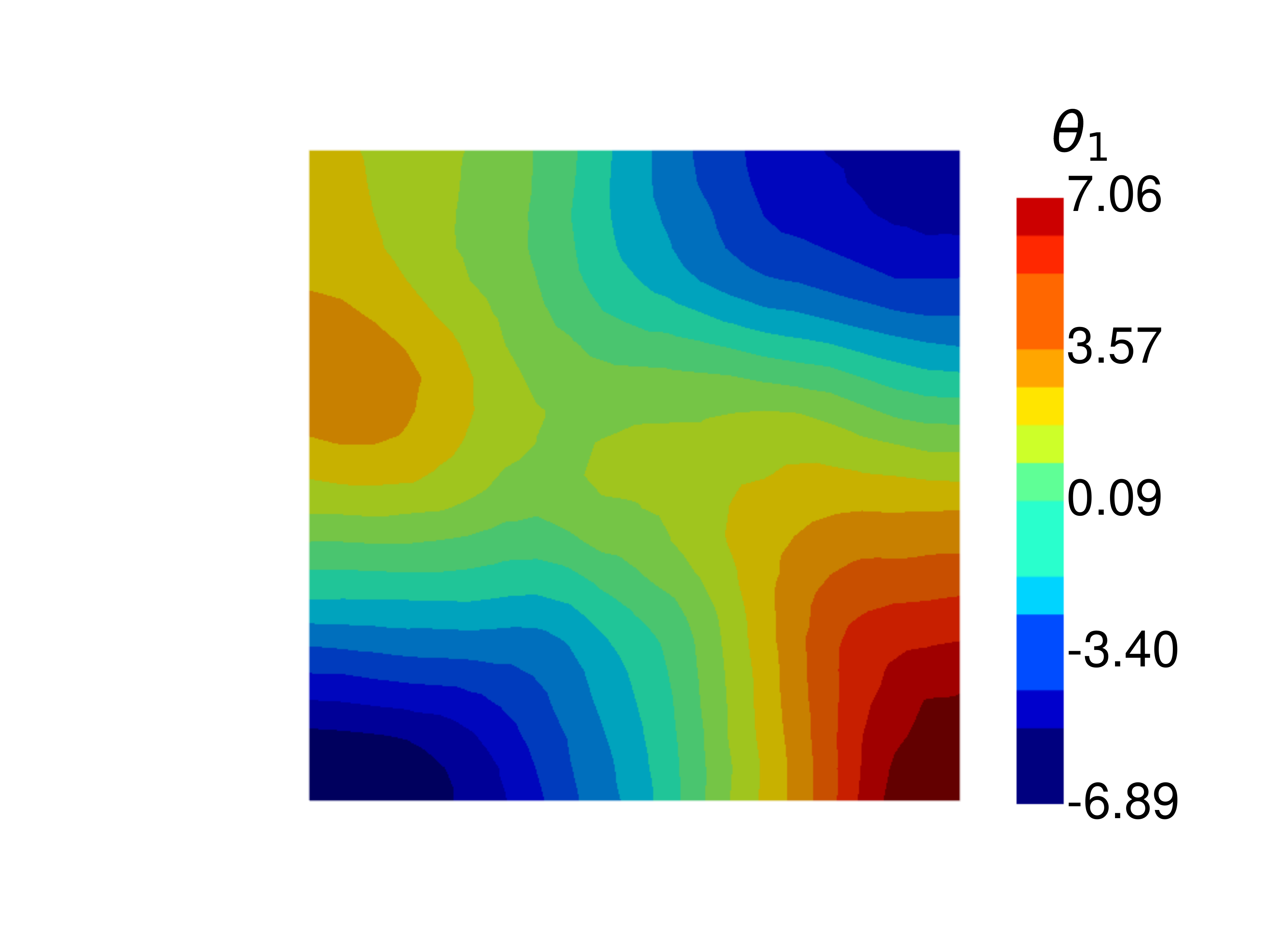} \caption{}\label{fig:mcmc:mlklpde-10:grf-l1}
\end{subfigure}
\begin{subfigure}[h]{0.24\textwidth} %
\includegraphicsifexists[trim={4.5cm 2cm 2.75cm 2cm},clip,width=\textwidth]{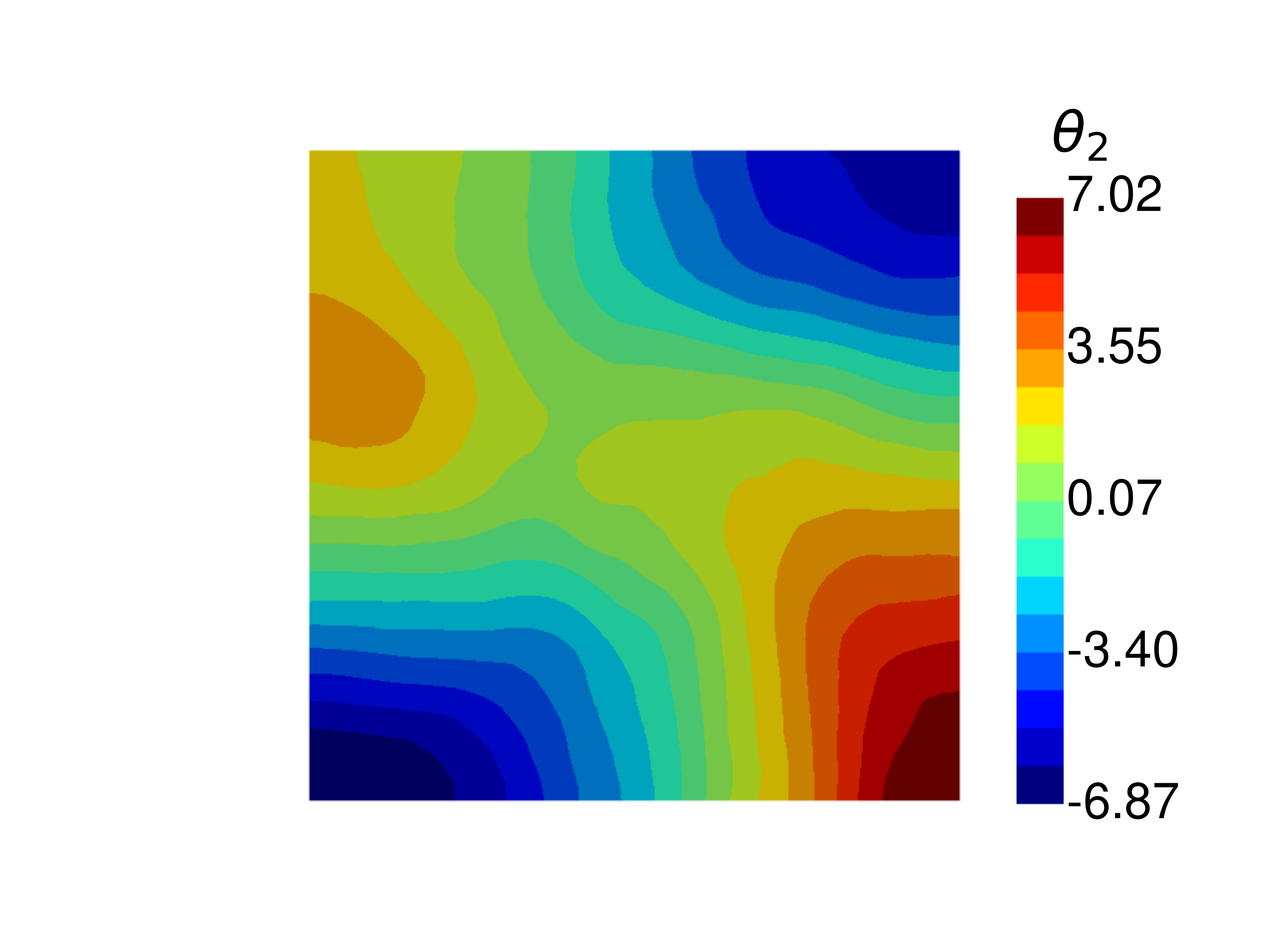} \caption{}\label{fig:mcmc:mlklpde-10:grf-l2}
\end{subfigure}
\begin{subfigure}[h]{0.24\textwidth} %
\includegraphicsifexists[trim={4.5cm 2cm 2.75cm 2cm},clip,width=\textwidth]{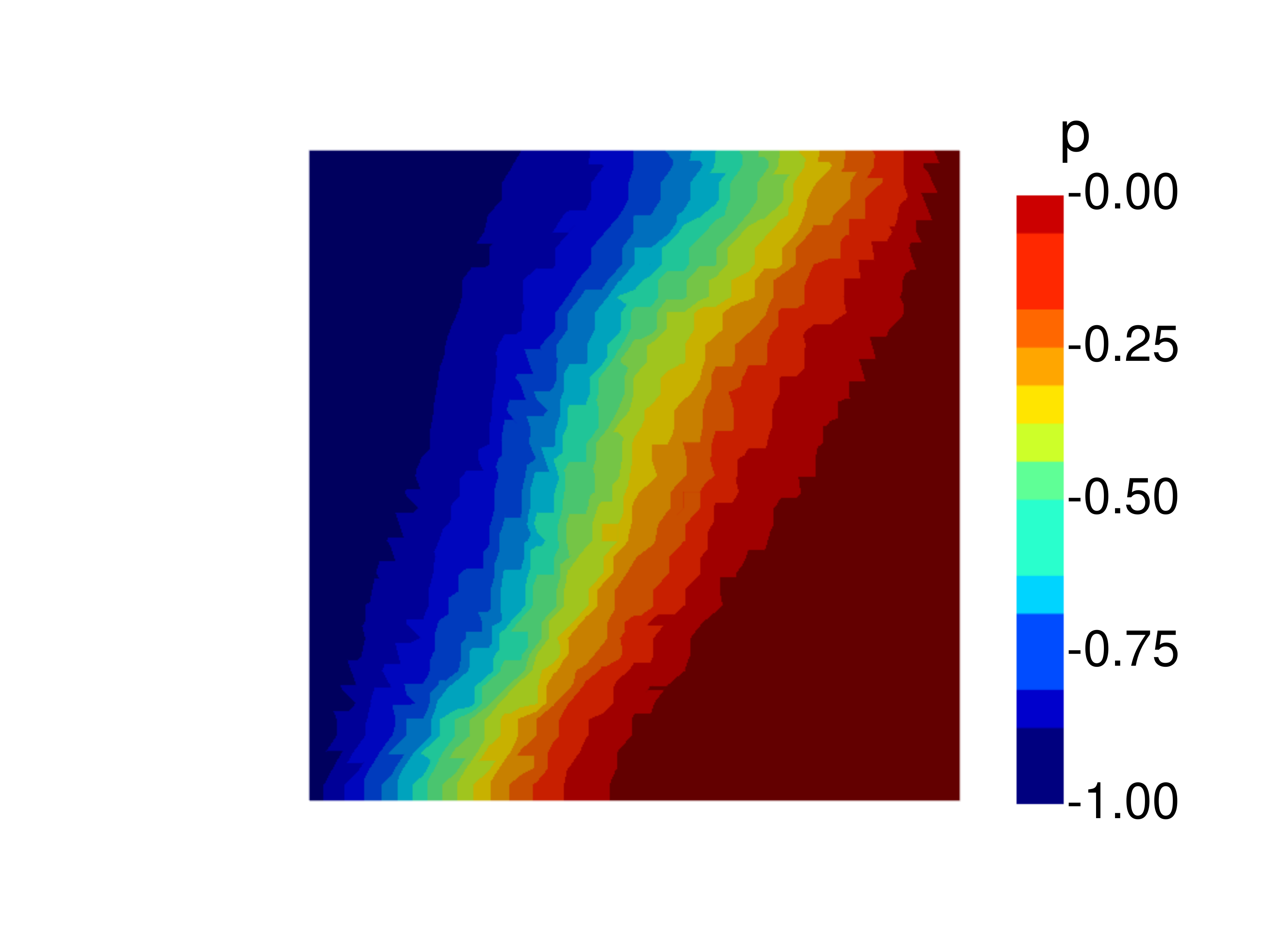} \caption{}\label{fig:mcmc:mlklpde-10:p-l2}
\end{subfigure}
\begin{subfigure}[h]{0.24\textwidth} %
\includegraphicsifexists[trim={4.5cm 2cm 2.75cm 2cm},clip,width=\textwidth]{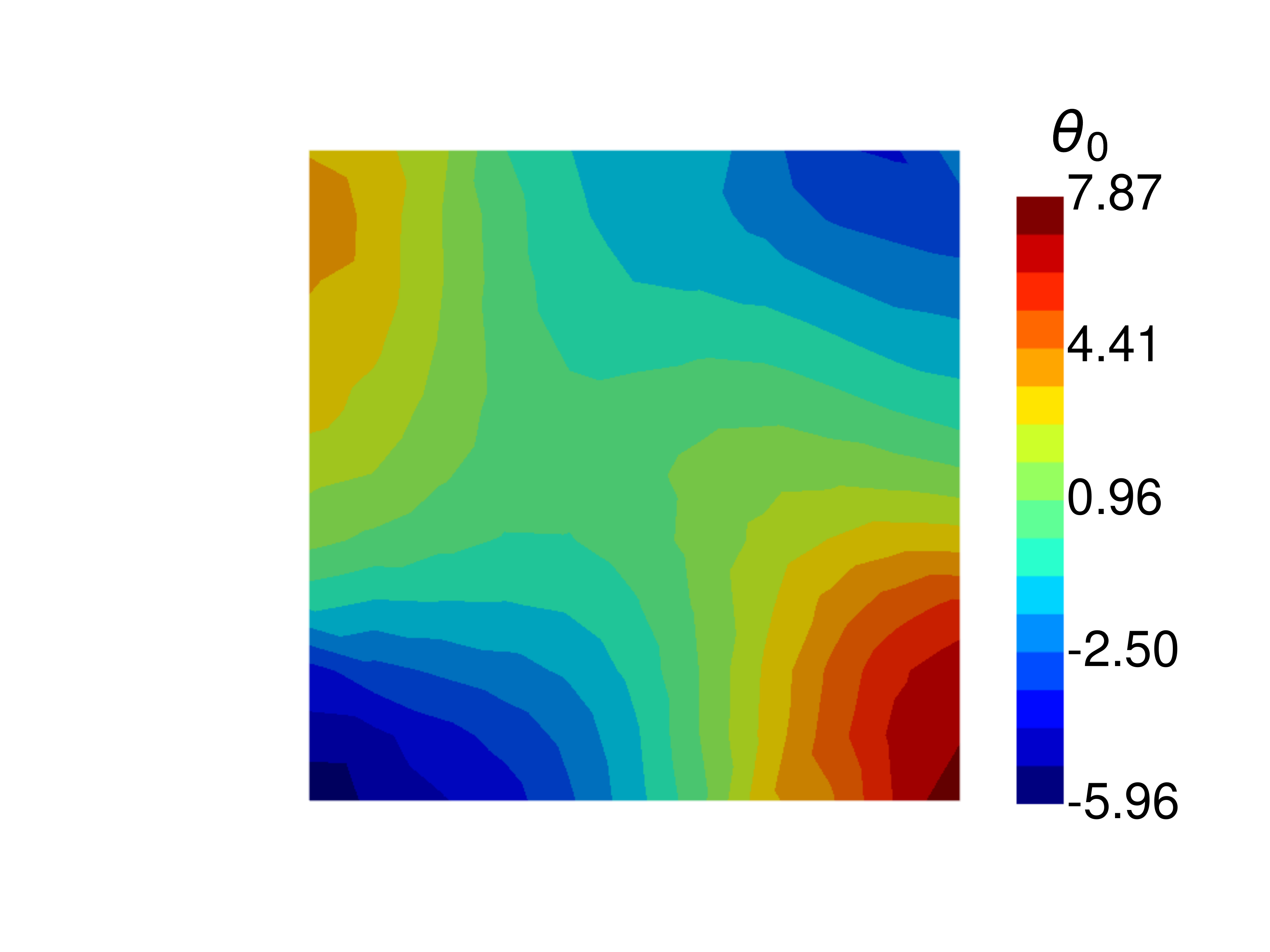} \caption{}\label{fig:mcmc:mlklpde-30:grf-l0}
\end{subfigure}
\begin{subfigure}[h]{0.24\textwidth} %
\includegraphicsifexists[trim={4.5cm 2cm 2.75cm 2cm},clip,width=\textwidth]{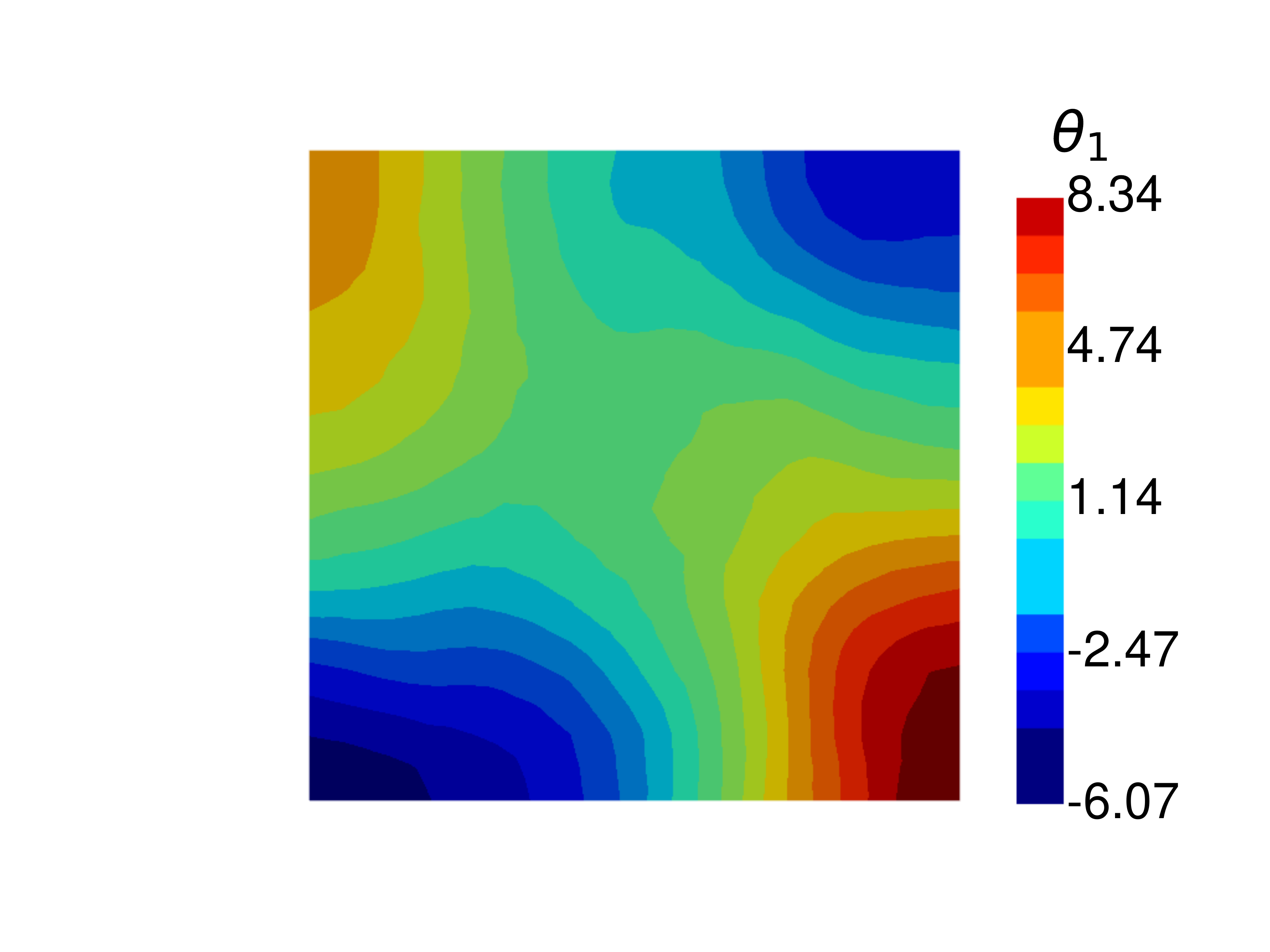} \caption{}\label{fig:mcmc:mlklpde-30:grf-l1}
\end{subfigure}
\begin{subfigure}[h]{0.24\textwidth} %
\includegraphicsifexists[trim={4.5cm 2cm 2.75cm 2cm},clip,width=\textwidth]{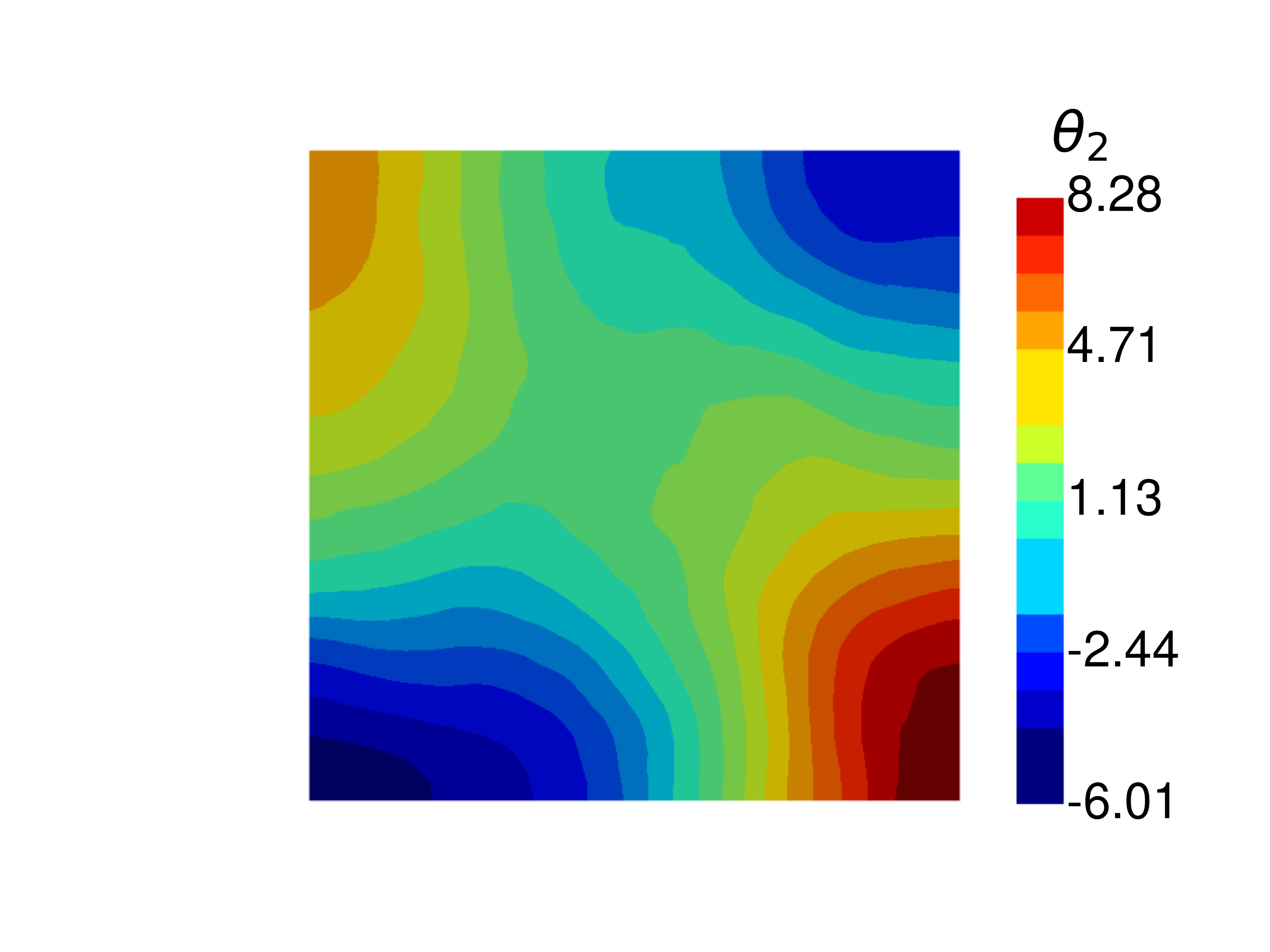} \caption{}\label{fig:mcmc:mlklpde-30:grf-l2}
\end{subfigure}
\begin{subfigure}[h]{0.24\textwidth} %
\includegraphicsifexists[trim={4.5cm 2cm 2.75cm 2cm},clip,width=\textwidth]{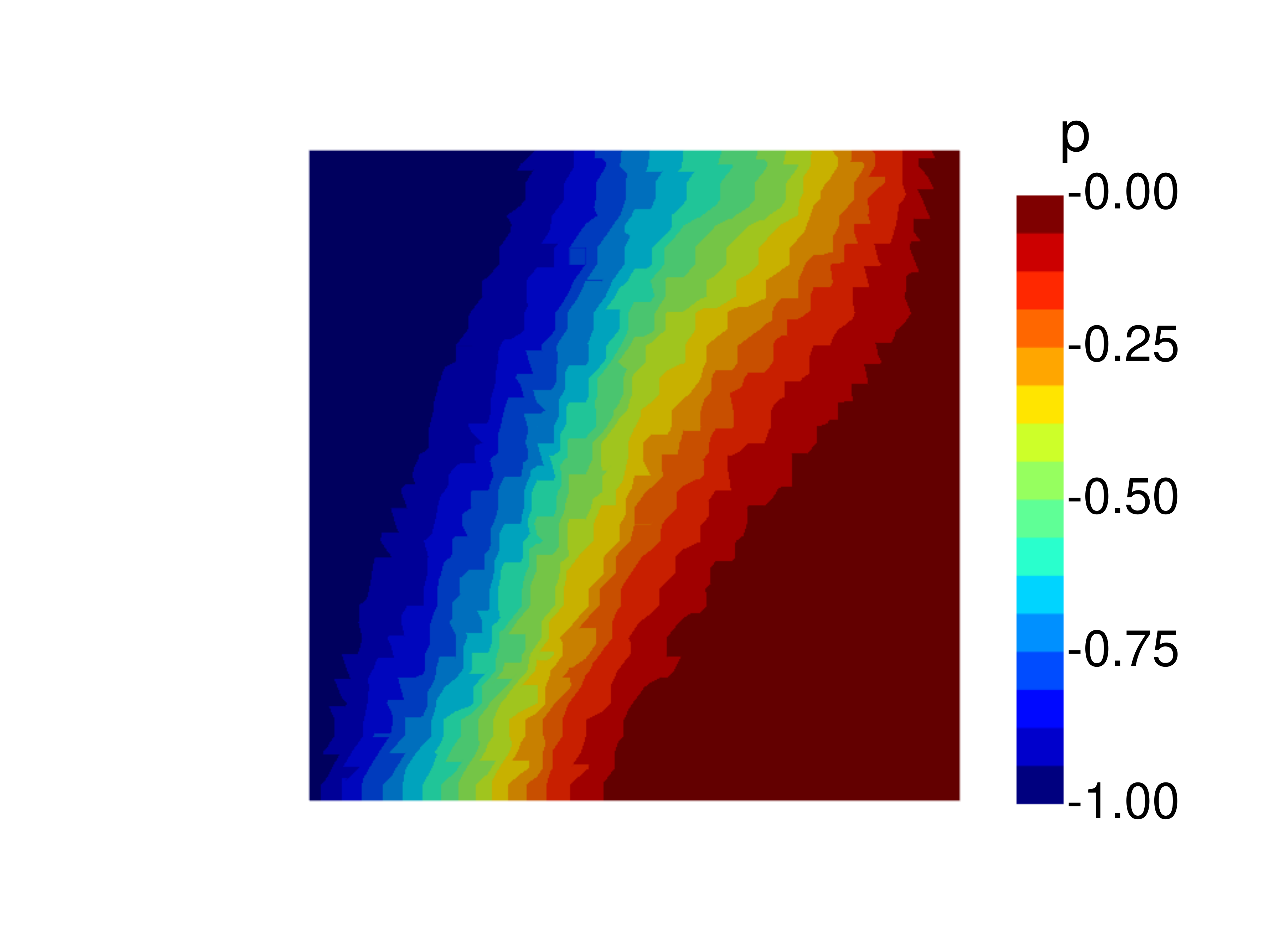} \caption{}\label{fig:mcmc:mlklpde-30:p-l2}
\end{subfigure}
\captionsetup{singlelinecheck=off,font=footnotesize}
\caption[]{Realizations of GRF samples on different levels and corresponding pressure fields obtained via multilevel using SPDE (top), KL with 10 modes (middle) and KL with 50 modes (bottom).}
\label{fig:mcmc:runs}
\end{figure}

Figure \ref{fig:mcmc:iact} shows the autocorrelation of various multilevel sampling configurations with three levels in the hierarchy. For each configuration, the chains for $Q_\ell$ on levels $\ell=1$ and $\ell=2$ all demonstrate similar autocorrelation, attributed to the insignificant differences between the fields on finer levels. The mixing in the chain for $Y_\ell$ improves for finer levels and is similar for both the SPDE and KL-SPDE samplers. The mixing of the KL-SPDE does not monotonically improve with the number of modes used, as the KL-SPDE sampler with 50 modes demonstrates better mixing than the KL-SPDE sampler with 10 or 75 modes. However, compared to the SPDE sampler, the KL-SPDE sampler with 50 modes demonstrates better mixing in the chain for $Q_\ell$ and can be attributed to the reduced dimensionality of the problem on the coarsest level. This suggests that there exists an optimal number of modes (i.e. a coarse subspace) that balances the dimensionality reduction, the accuracy of the representation and efficiency of the sampler. When sampling using the KL-SPDE sampler with $\mathpzc{m_0}$ modes, the sample space dimensionality on the subsequent finer level is $\mathpzc{n_1} - \mathpzc{m_0}$ (i.e. the dimensionality of the complement space). Hence, using fewer modes on the coarse level increases the dimensionality of the subsequent, more expensive fine level. Although it may appear that using more/all of the modes (i.e. $\mathpzc{m_0}=\mathpzc{n_0}$) may decrease the cost of the fine level, incorporating higher frequency modes into the expansion introduces aliasing errors, since frequencies above the Nyquist frequency cannot be resolved \cite{Schmidt2020}. This aliasing error pollutes the sampling on the coarse level which in turn affects the sampling on the finer levels.

\begin{figure}[H] \centering
\begin{subfigure}[h]{0.24\textwidth} %
\includegraphicsifexists[width=\textwidth]{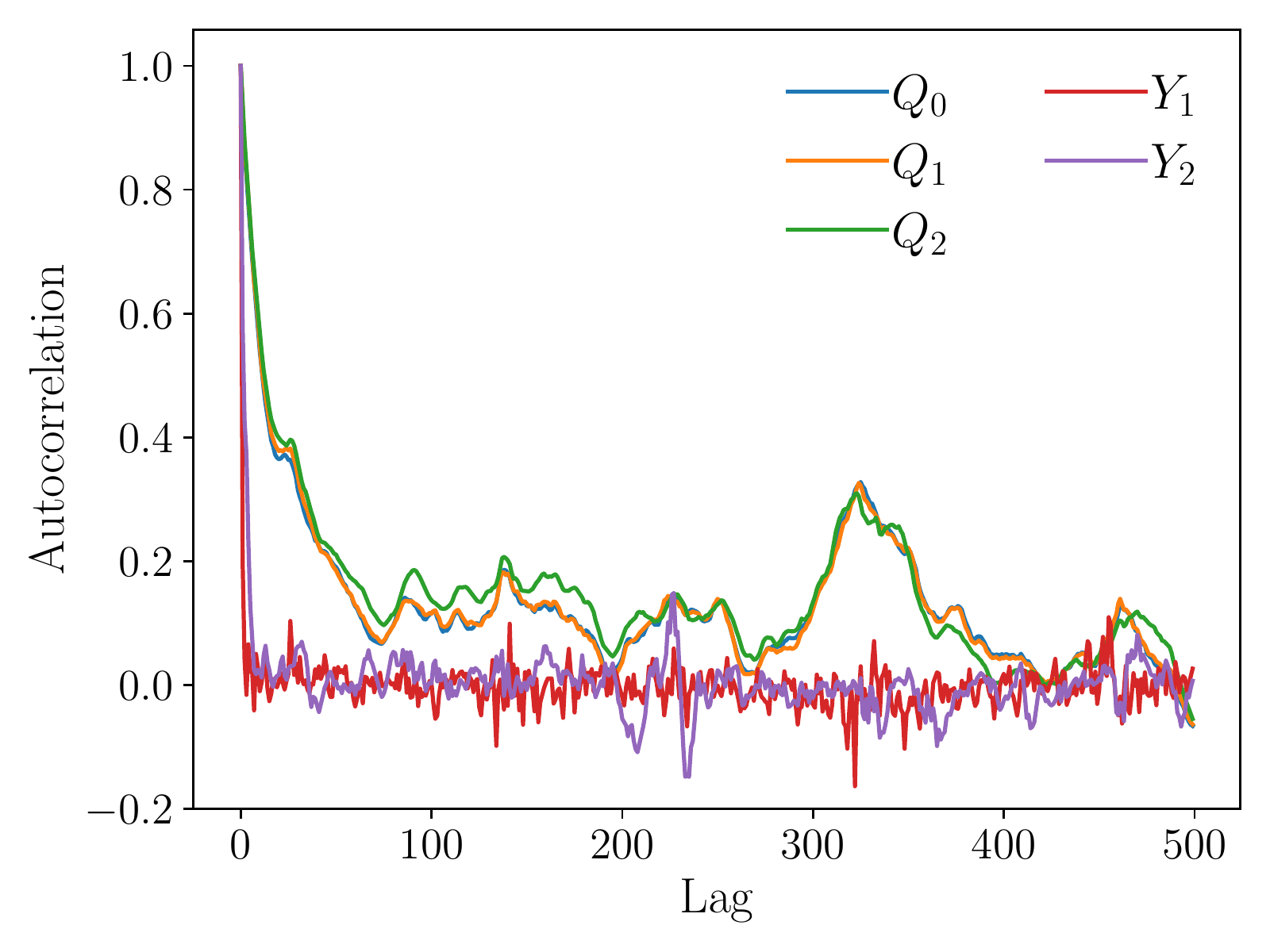} \caption{}\label{fig:mcmc:mlpde:iact}
\end{subfigure}
\begin{subfigure}[h]{0.24\textwidth} %
\includegraphicsifexists[width=\textwidth]{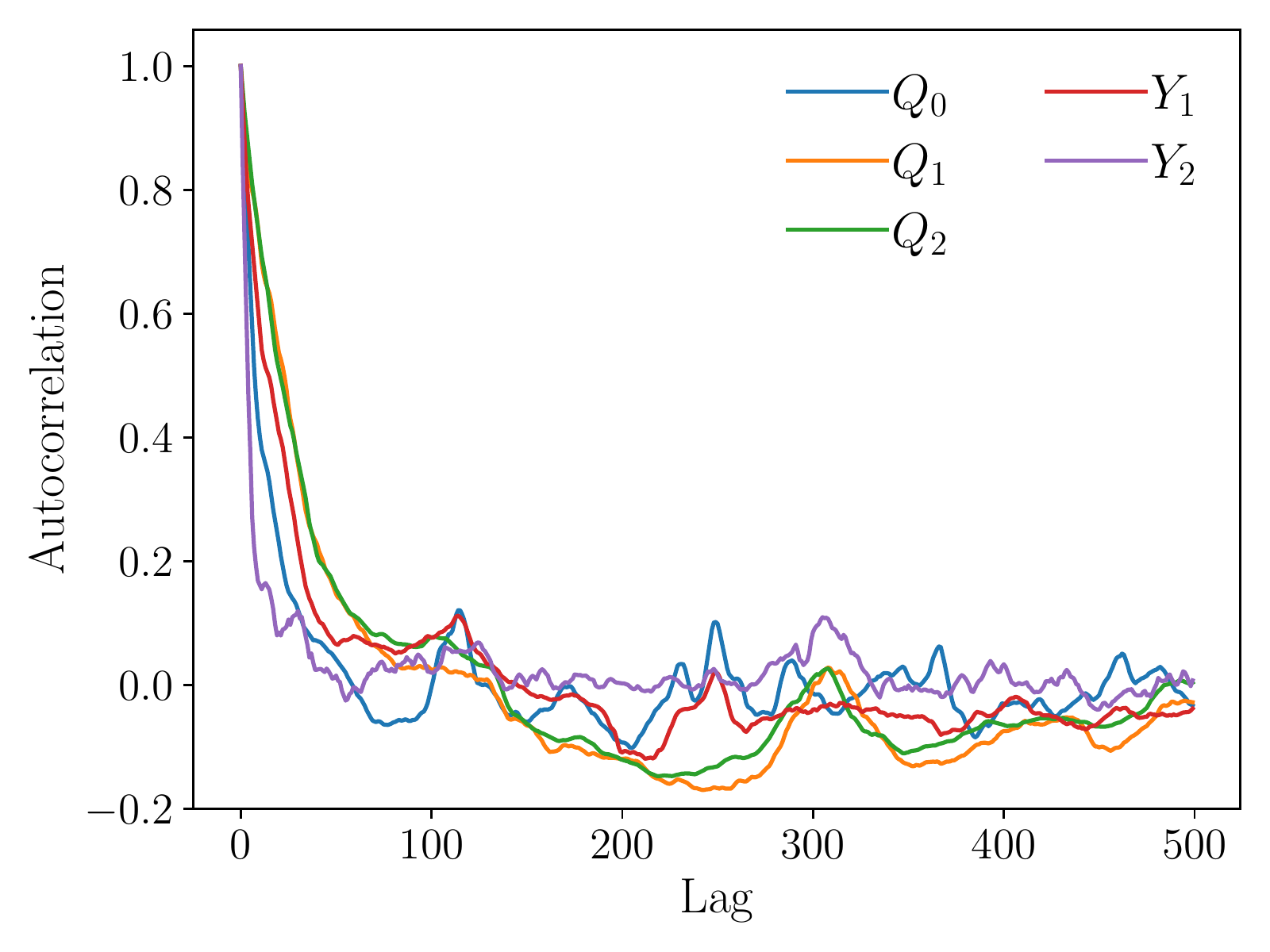} \caption{}\label{fig:mcmc:mlklpde-10:iact}
\end{subfigure}
\begin{subfigure}[h]{0.24\textwidth} %
\includegraphicsifexists[width=\textwidth]{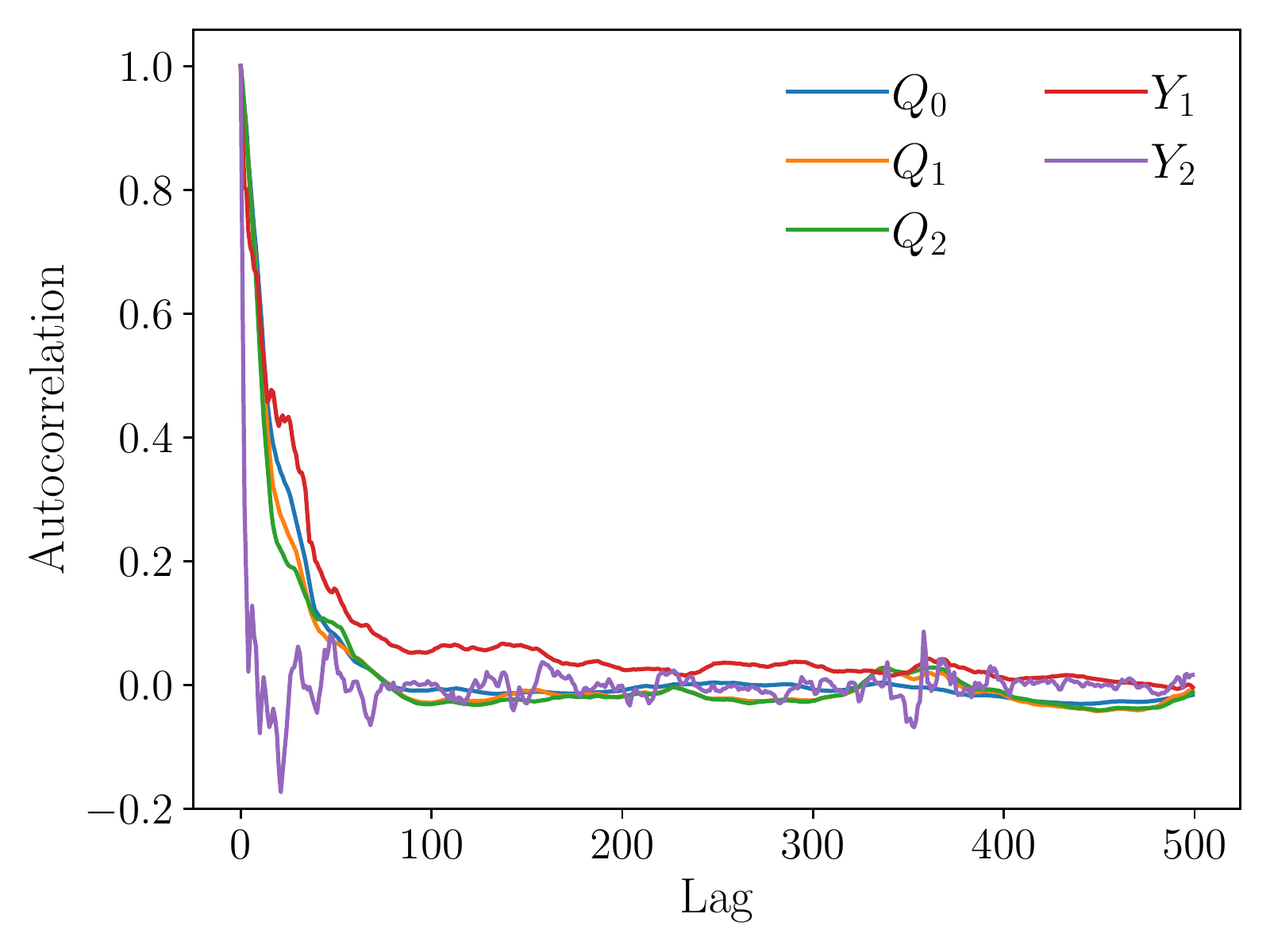} \caption{}\label{fig:mcmc:mlklpde-50:iact}
\end{subfigure}
\begin{subfigure}[h]{0.24\textwidth} %
\includegraphicsifexists[width=\textwidth]{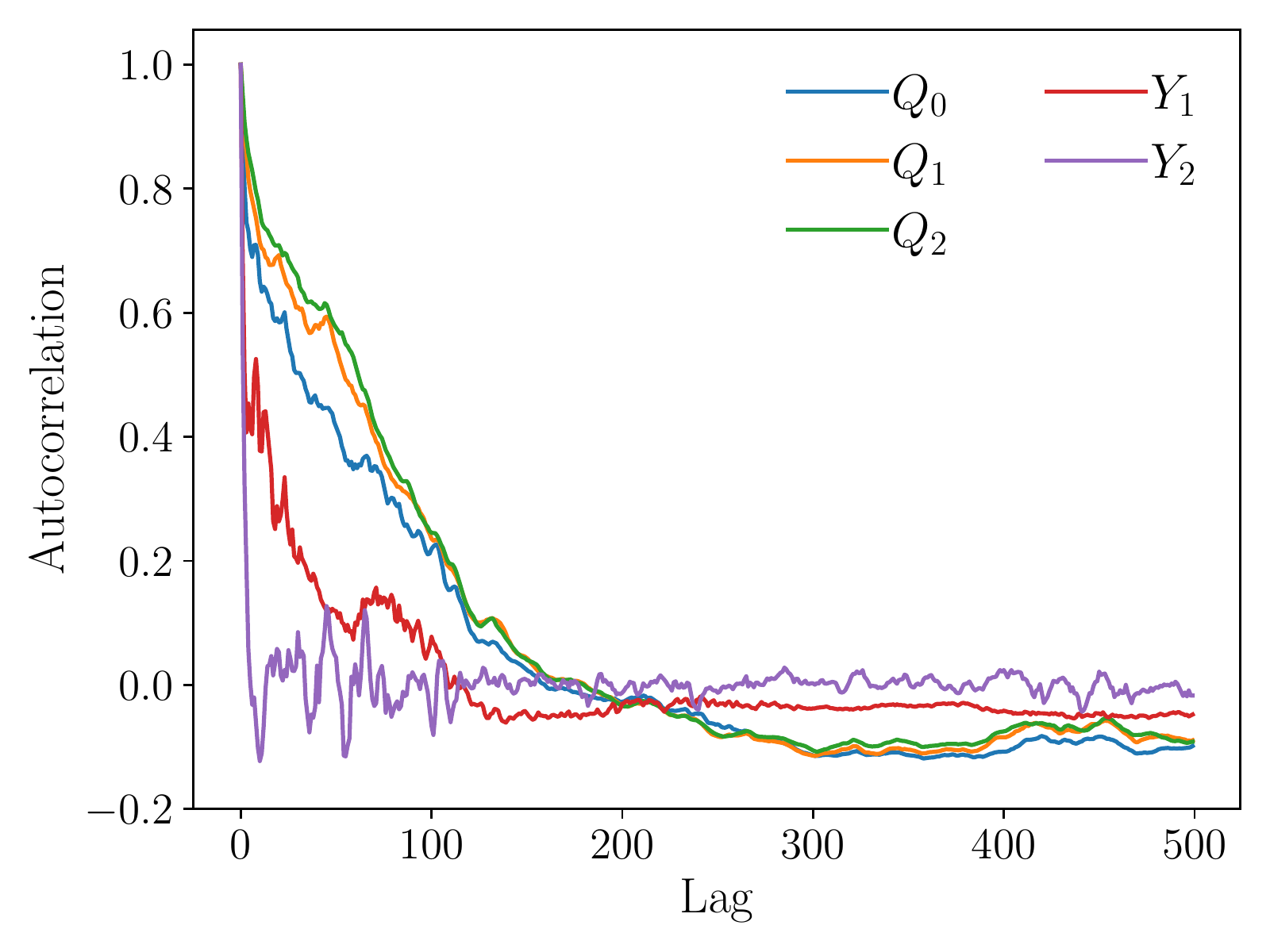} \caption{}\label{fig:mcmc:mlklpde-75:iact}
\end{subfigure}
\captionsetup{singlelinecheck=off,font=footnotesize}
\caption[]{Autocorrelation function of $Y$ and $Q$ on different levels obtained using: a) SPDE sampling, and KL-SPDE sampling with: b) 10 modes, c) 50 modes and d) 75 modes.}
\label{fig:mcmc:iact}
\end{figure}

Table \ref{table:3Levels:IACT} shows the coarse level variance, IACT, effective sample size and total cost to achieve a tolerance of $\varepsilon = 0.1$ for various multilevel sampling configurations with a three level hierarchy. We see that the lowest coarse-level variance, smallest effective sample size and lowest total cost is achieved using 50 KL modes on the coarsest level, with the cost increasing when using a smaller or larger number of KL modes. This again suggests the existence of an optimal number of modes that balances the dimensionality reduction and aliasing errors. Increasing the number of KL modes towards this \emph{optimum} value leads to a rapid decrease (by a factor of $\approx 3$-per 25 KL modes) in the total cost while increasing beyond this value leads to an increase in the total cost (by a factor of $\approx 2$-per 25 KL modes). This indicates that the rate of performance improvement from increasing dimensionality of the coarse-level sample space is greater than the performance degradation caused by aliasing errors.

\begin{table}[H]
\caption{Average (over five chains) effective sample size and total cost to achieve tolerance of $\varepsilon = 0.1$ for various multilevel sampling configurations with three levels in the hierarchy. Here, $\mathrm{KL}(\mathpzc{m})$ refers to the KL-SPDE sampler with $\mathpzc{m}$ KL modes.} \label{table:3Levels:IACT}
\centering
\begin{tabular}{ l | c | c c c | c c c | c } 
\hline 
\hline 
&  & \multicolumn{3}{c|}{IACT} & \multicolumn{3}{c|}{Effective Sample Size} & Total Cost (s) \\
\hline
Configuration  & $\Var[\pi_0]{Q_0}$ & $\ell=0$ & $\ell=1$ & $\ell=2$ &  $\ell=0$ & $\ell=1$ & $\ell=2$  & \\ 
\hline 
SPDE  & 33 & 223 & 21 & 49 &  11651 & 544 & 597 & 33637 \\
KL(10) & 43 & 275 & 81 & 33 & 26431 & 297 & 203 & 186555 \\
KL(25)  & 63 & 271 & 16 & 60 & 21800 & 644 & 363 & 68658 \\
KL(50)  & 19 & 276 & 27 & 54 & 7653 & 254 & 192 & 28800 \\
KL(75)  & 31 & 275 & 28 & 80 & 11774 & 510 & 217 & 42333 \\
KL(100) & 44 & 271 & 42 & 30 & 16967 & 330 & 475 & 65005 \\
\hline 
\hline
\end{tabular}
\end{table}

Figure \ref{fig:mcmc:acceptance} shows the acceptance rate of the samplers on each level of a four level hierarchy. It clearly demonstrates that the acceptance rate monotonically increases with refinement and is aligned with MLMCMC theory \cite{Dodwell2015}. The acceptance rate of the KL-SPDE samplers tends to that of the SPDE sampler with an increasing number of KL modes. It should be mentioned that the acceptance rate of the sampling algorithms is also influenced by several factors including the tuning parameters (e.g., $\beta$ in pCN proposal) in the MCMC sampler, and hence, can be optimized for each sampler and test problem. The mean and variance of $Y_\ell$ on different levels are shown in Figures \ref{fig:mcmc:meanY} and \ref{fig:mcmc:varY}, respectively. The mean of $|Y_\ell|$ and variance of $Y_\ell$ both monotonically decay with refinement, with the KL-SPDE sampler demonstrating both lower values and faster initial decay, with subsequent finer levels displaying similar decay rates as the SPDE sampler. The mean and variance of $Y_\ell$ is dependent on the number of modes used but is lower than that of the SPDE sampler when using a sufficient number of KL modes. The faster decay in the mean of $|Y_\ell|$ yields a faster convergence in the multilevel telescoping sum (\eqref{ML:Expectation}) while a faster decay in the variance of $Y_\ell$ yields a faster convergence rate of the Monte Carlo estimator of $\mathbb{E}_{\pi_\ell,\pi_{\ell-1}}[Y_\ell]$ (\eqref{ML:Y}). Figures \ref{fig:mcmc:iact} and \ref{fig:mcmc:stats}, and Table \ref{table:3Levels:IACT} all show that the KL-SPDE sampler with 50 modes on the coarsest level demonstrate lower variances, cost, effective sample sizes and improved mixing on each level, and strongly indicates the existence of an optimal subspace. This configuration balances the dimensionality reduction and aliasing errors, leading to an efficient sampler for the problem. Of course, the parameters of the samplers (e.g., proposal variance) can be calibrated to further improve the efficiency of the samplers.

\begin{figure}[H] \centering
\begin{subfigure}[h]{0.33\textwidth} %
\includegraphicsifexists[width=\textwidth]{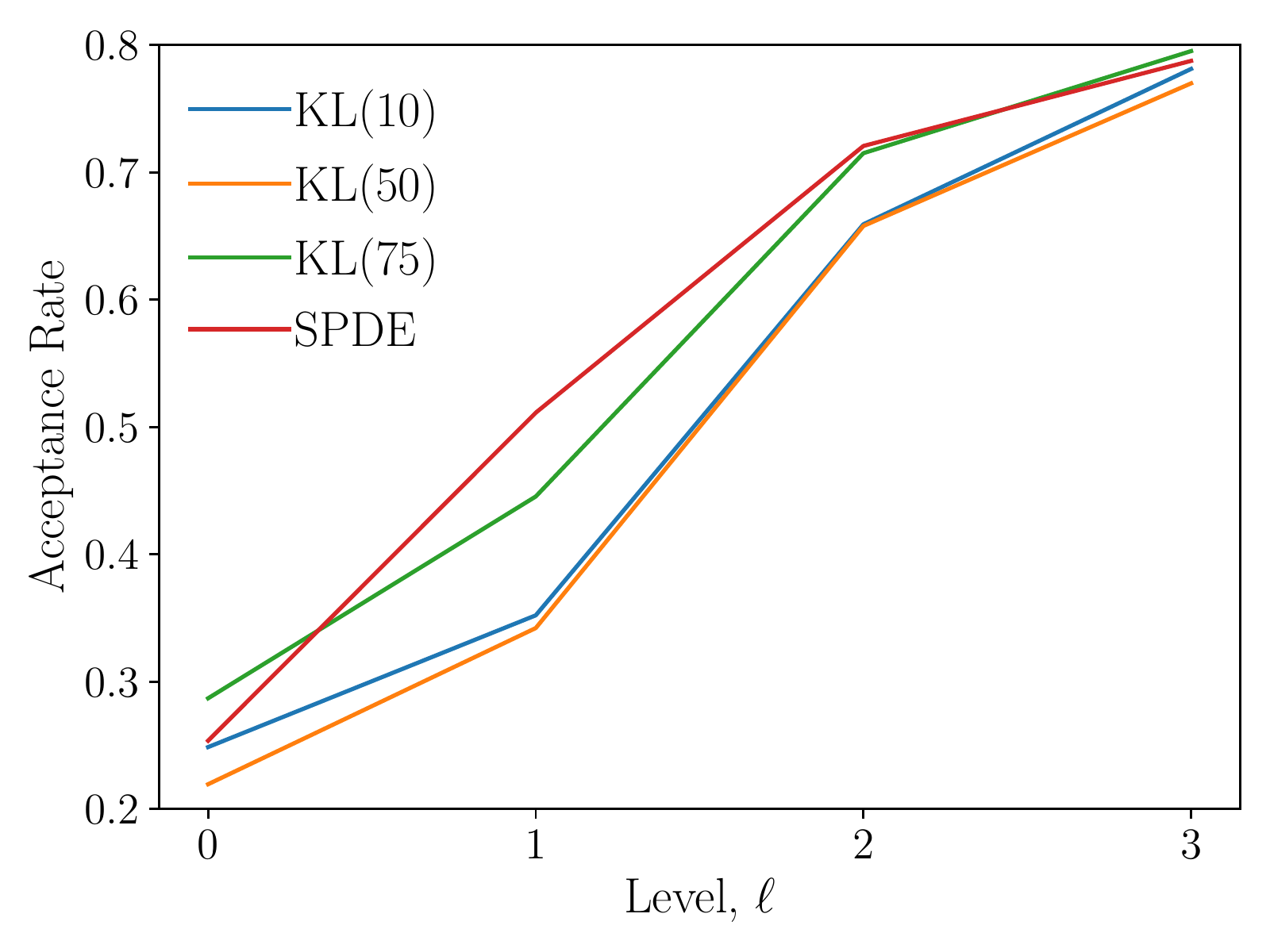} \caption{}\label{fig:mcmc:acceptance}
\end{subfigure}
\begin{subfigure}[h]{0.33\textwidth} %
\includegraphicsifexists[width=\textwidth]{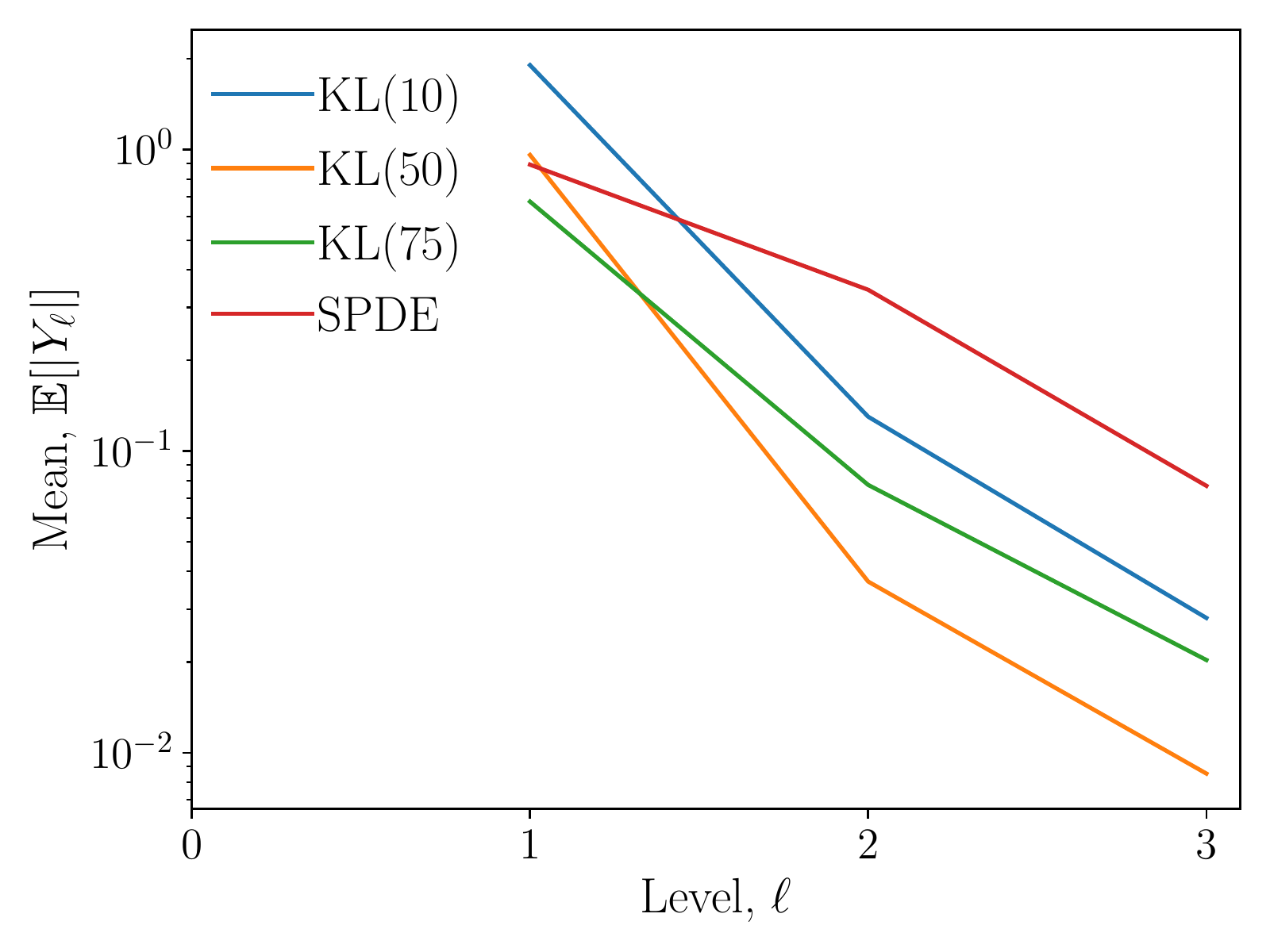} \caption{}\label{fig:mcmc:meanY}
\end{subfigure}
\begin{subfigure}[h]{0.33\textwidth} %
\includegraphicsifexists[width=\textwidth]{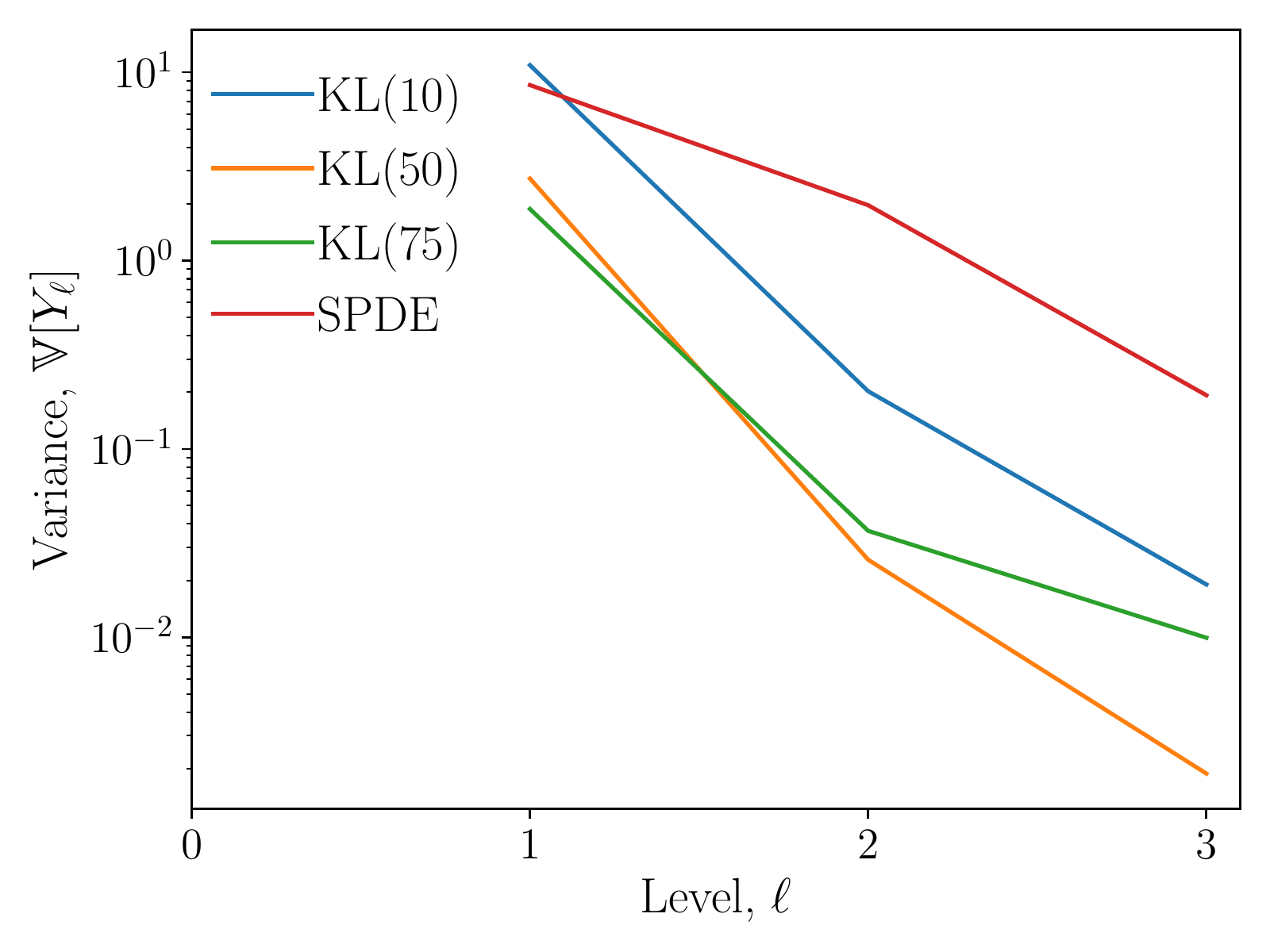} \caption{}\label{fig:mcmc:varY}
\end{subfigure}
\captionsetup{singlelinecheck=off,font=footnotesize}
\caption[]{a) Acceptance rate, b) mean of $|Y_\ell|$, and c) variance of $Y_\ell$ on different levels obtained using different multilevel sampling configurations. Here, $\mathrm{KL}(\mathpzc{m})$ refers to the KL-SPDE sampler with $\mathpzc{m}$ KL modes.}
\label{fig:mcmc:stats}
\end{figure}

\section{Conclusion} \label{sec:conclusion}

In this work, we presented multilevel decompositions for hierarchical Markov Chain Monte Carlo (MCMC) sampling, designed to enhance efficiency while maintaining accuracy in high-dimensional inference problems. Our approach integrates multilevel decomposition techniques based on stochastic partial differential equation (SPDE) with a Karhunen–Loève (KL) expansion to optimize the sampling process. A key feature of this framework is its ability to reduce sample space while preserving discretization accuracy and the ergodicity of the MCMC sampler, ensuring convergence and reliable sampling across all levels of the hierarchy.

We demonstrated the effectiveness of this method through its application on an inference problem in groundwater flow, a challenging problem characterized by high-dimensional and uncertain parameter spaces. Our results revealed that the optimal KL expansion applied at the coarsest level reduces sampling cost, the variance of the estimator, and the mixing of the MCMC chain. This reduction in cost and variance is propagated to finer levels, leading to more efficient and accurate sampling across the hierarchy. Importantly, we also demonstrated, empirically, that there exists an optimal number of KL modes at the coarse level (i.e. a coarse subspace), striking a balance between dimensionality reduction that reduces the sample space size and aliasing errors that pollute the GRF samples. Additionally, the use of KL decomposition maintained a similar rate of decay in the variance across levels while achieving variance on each level, further enhancing efficiency.

Future work will focus on identifying the optimal coarse subspace more rigorously and exploring the use of machine learning-based approximations to further accelerate computations of the forward problem at individual levels of the hierarchy. These advancements have the potential to further enhance the scalability and applicability of the multilevel MCMC framework, enabling more efficient uncertainty quantification in increasingly complex systems.

\section*{Acknowledgements}

This work was performed under the auspices of the U.S. Department of Energy by Lawrence Livermore National Laboratory under Contract DE-AC52-07NA27344. LLNL-JRNL-872668

\bibliographystyle{unsrt}
\bibliography{references}  %

\end{document}